\documentclass[11pt,a4paper]{article}

\title{Automorphism groups of countable algebraically closed graphs
  and endomorphisms of the random graph}
\author{Igor Dolinka\thanks{%
    Department of Mathematics and Informatics, University of Novi Sad,
    Trg Dositeja Obradovi\'{c}a~4, 21101 Novi Sad, Serbia},
  Robert~D. Gray\thanks{%
    School of Mathematics, University of East Anglia, Norwich NR4 7TJ,
    United Kingdom},
  Jillian~D. McPhee\thanks{%
    School of Mathematics \& Statistics, University of St Andrews, St
    Andrews, Fife KY16 9SS, United Kingdom},\\
  James~D. Mitchell$^{\ddagger}$ \& Martyn Quick$^{\ddagger}$}

\usepackage{amsmath,amssymb}
\usepackage{mathrsfs}
\usepackage[pdf]{pstricks}

\usepackage{geometry}
\usepackage{color}
\usepackage[colorlinks]{hyperref}
\hypersetup{linkcolor=blue, urlcolor=blue, citecolor=red}

\usepackage{theorem}
\theorembodyfont{\slshape}
\newtheorem{thm}{Theorem}[section]
\newtheorem{cor}[thm]{Corollary}
\newtheorem{lemma}[thm]{Lemma}
\newtheorem{prop}[thm]{Proposition}
\theorembodyfont{\normalfont}
\newtheorem{comment}[thm]{Comment}

\newcommand{\qed}{\hspace*{\fill}$\square$}
\newenvironment{proof}{%
  \begin{trivlist}\item\textsc{Proof:}}{\qed\end{trivlist}}

\newcommand{\1}{\mathbf{1}}
\newcommand{\Aut}{\operatorname{Aut}}
\newcommand{\End}{\operatorname{End}}
\newcommand{\disjunion}{\mathbin{\dot{\cup}}}
\newcommand{\disjsqunion}{\mathbin{\dot{\sqcup}}}
\newcommand{\im}{\operatorname{im}}
\newcommand{\N}{\mathbb{N}}
\newcommand{\Q}{\mathbb{Q}}
\newcommand{\Schutz}{Sch\"{u}tzen\-berger}
\newcommand{\scD}{\mathscr{D}}
\newcommand{\scH}{\mathscr{H}}
\newcommand{\scJ}{\mathscr{J}}
\newcommand{\scL}{\mathscr{L}}
\newcommand{\scR}{\mathscr{R}}
\newcommand{\Sym}{\operatorname{Sym}}
\newcommand{\nbd}{\nobreakdash-}

\newcommand{\mathbit}[1]{\mbox{\boldmath$#1$}}
\newcommand{\order}[1]{\mathopen{|}#1\mathclose{|}}
\newcommand{\set}[2]{\{\,#1\mid#2\,\}}

\newcommand{\spc}{\vspace{\baselineskip}}

\renewcommand{\geq}{\geqslant}
\renewcommand{\leq}{\leqslant}
\renewcommand{\emptyset}{\varnothing}

\makeatletter
\newcommand{\blankfoot}[1]{%
  \xdef\@thefnmark{}\@footnotetext{#1}}
\makeatother

\psset{unit=30pt}

\begin{document}

\maketitle

\blankfoot{\emph{Keywords:} existentially closed graphs, algebraically
  closed graphs, random graph, endomorphism monoid, countable universal
  graph, countable universal bipartite graph}
\blankfoot{\emph{MSC:} 05C25, 03C50, 20M20, 20B27}

\begin{abstract}
  We establish links between countable algebraically closed graphs and
  the endomorphisms of the countable universal graph~$R$.  As a
  consequence we show that, for any countable graph~$\Gamma$, there
  are uncountably many maximal subgroups of the endomorphism monoid
  of~$R$ isomorphic to the automorphism group of~$\Gamma$.  Further
  structural information about~$\End R$ is established including that
  $\Aut \Gamma$~arises in uncountably many ways as a \Schutz\ group.
  Similar results are proved for the countable universal directed
  graph and the countable universal bipartite graph.
\end{abstract}

\section{Introduction}

Existentially closed relational structures have been widely considered
with the example of the \emph{countable universal homogeneous graph}
(also known as the \emph{random graph} or the \emph{Rado graph})
probably the most studied (see, for one example of a
survey,~\cite{CameronSurvey}).  It was established by
Truss~\cite{Truss} that the automorphism group of the countable
universal homogeneous graph is simple and this was placed in a general
setting by Macpherson and Tent~\cite{MacT}.  The work in the present paper
arose when attempting to establish what can be said about other
naturally arising groups acting (in some sense) upon the countable
universal graph~$R$.  To be more precise, we present information about
the maximal subgroups of the endomorphism monoid of~$R$.  We note that
this is not the first work to focus on endomorphisms in the context of
homogeneous structures.  For example, Cameron and
Ne\v{s}et\v{r}il~\cite{CamNes} consider homomorphism-homogeneous
structures and there are various links between their results and our
work, particularly \cite[Section~2]{CamNes}.  More recently, Lockett
and Truss~\cite{LocTru} examine generic endomorphisms of homogeneous
structures and in their concluding remarks propose that there should
be a counterpart to the literature on automorphism groups of such
structures applying to the monoids of endomorphisms.  This paper may
be thought of as part of the study suggested by Lockett and Truss.

A maximal subgroup of the endomorphism monoid of~$R$ is determined by
the idempotent endomorphism that plays the role of its identity
element.  Indeed, it is the $\scH$\nbd class of that endomorphism, as
we summarise in Section~\ref{sec:prelims} below.  Bonato and
Deli\'{c}~\cite[Proposition~4.2]{BonDel} show that the images of
idempotent endomorphisms of the countable universal graph are
characterised as being \emph{algebraically closed} (a property weaker
than existentially closed) and this is discussed in detail by
Dolinka~\cite[especially Theorem~3.2]{Dolinka1}.  We make the same
observation and also the corresponding result for images of idempotent
endomorphisms of the countable universal directed graph and countable
universal bipartite graph in the course of our work.  Indeed, we
observe that there are, except for one case,
$2^{\aleph_{0}}$~idempotent endomorphisms with image isomorphic to a
given algebraically closed graph, directed graph, or bipartite graph
(Theorems~\ref{thm:graph-idemp}, \ref{thm:directed-idemp}
and~\ref{thm:bipartite-idemp}, respectively).  These observations are,
however, merely the first steps in establishing the results herein.

We shall establish the same types of theorem for the classes of graphs
(that is, undirected graphs), directed graphs, and bipartite graphs.
These classes of relational structure are treated in turn in separate
sections below.  The proofs for (undirected) graphs are the archetypes
and so the sections relating to directed graphs and bipartite graphs
are concerned mostly with explaining what modifications are required
to establish the analogous results.  One needs particular care with
bipartite graphs, in the first instance to ensure that the correct
definition is chosen so that the class of bipartite graphs does indeed
have a Fra\"{\i}ss\'{e} limit, as noted in~\cite{Evans}.  However, a
second wrinkle occurs since there are examples of algebraically closed
bipartite graphs that are finite (for example, the complete bipartite
graph~$K_{m,n}$ on two parts of cardinality $m$~and~$n$ respectively),
unlike the situation for graphs and directed graphs where
algebraically closed structures are necessarily infinite, and this has
some surprising consequences for our results (compare
Theorems~\ref{thm:bipartite-idemp} and~\ref{thm:bipartite-regLR} with
their graph analogues).  We therefore need to introduce a stronger
condition, that we term \emph{strongly algebraically closed}, in order
to establish some of the analogues for the countable universal
bipartite graph.  These issues are discussed in detail in
Section~\ref{sec:bipartite}.  Homogeneous bipartite graphs were, for
example, also considered by Goldstern, Grossberg and
Kojman~\cite{GGK}, but they only permit what we term part-fixing
automorphisms whereas our automorphisms will be allowed to interchange
the parts.

In the summary of our results that follows, we use the term ``any
group'' to mean a group isomorphic to the automorphism group of a
countable graph.  The extension of Frucht's Theorem~\cite{Frucht} to
infinite groups established by de Groot~\cite{deGroot} and by
Sabidussi~\cite{Sabidussi} tells us this includes every countable
group.  We note in the course of our work that this class of groups is
the same as those arising as the automorphism group of countable
directed graphs (Proposition~\ref{prop:directed-automs}) and of
countable bipartite graphs (Theorem~\ref{thm:bipartite-automs}).

Let $\mathcal{C}$~denote either the class of countable graphs,
countable directed graphs, or countable bipartite graphs and let
$\Omega$~denote the universal homogeneous structure in~$\mathcal{C}$.
Then
\begin{itemize}
\item any group arises in $2^{\aleph_{0}}$~ways as the automorphism
  group of an algebraically closed structure in~$\mathcal{C}$
  (Theorems~\ref{thm:ac-graphs}, \ref{thm:ac-directed}
  and~\ref{thm:ac-bipartite});
\item any group arises in $2^{\aleph_{0}}$~ways as a maximal subgroup
  of the endomorphism monoid of~$\Omega$
  (Theorems~\ref{thm:R-reg-classes}, \ref{thm:D-reg-classes}
  and~\ref{thm:B-reg-classes});
\item any group arises in $2^{\aleph_{0}}$~ways as the \Schutz\ group
  of a non-regular $\scH$\nbd class in the endomorphism monoid
  of~$\Omega$ (Theorems~\ref{thm:R-uncountSchutz},
  \ref{thm:D-uncountSchutz} and~\ref{thm:B-uncountSchutz}).
\end{itemize}
The maximal subgroups of the endomorphism monoid of~$\Omega$ are the
group $\scH$\nbd classes of regular $\scD$\nbd classes of~$\End
\Omega$ (that is, $\scH$\nbd classes that inherit the structure of a
group from~$\End \Omega$).  For general $\scH$\nbd classes (including
all those in $\scD$\nbd classes that are not regular), there is an
alternative group that one can use instead.  This is the
\Schutz\ group referred to above (and which we expand upon in
Section~\ref{sec:prelims}) and generalises the concept of a group
$\scH$\nbd class (not least because the \Schutz\ group is isomorphic
to the $\scH$\nbd class when the latter happens to be a group).

Theorems~\ref{thm:R-reg-classes}, \ref{thm:D-reg-classes}
and~\ref{thm:B-reg-classes} say more about the structure of the
endomorphism monoid of~$\Omega$, namely every group arises as a group
$\scH$\nbd class in $2^{\aleph_{0}}$~many $\scD$\nbd classes and,
except for one case for the countable universal bipartite graph, every
regular $\scD$\nbd class contains $2^{\aleph_{0}}$ group $\scH$\nbd
classes.  From the first of these facts, it follows there are
$2^{\aleph_{0}}$ regular $\scD$\nbd classes in~$\End \Omega$.  We also
describe how many $\scL$- and $\scR$\nbd classes there are
(usually~$2^{\aleph_{0}}$) in each of these regular $\scD$\nbd classes
(see Theorems~\ref{thm:graph-regLR}, \ref{thm:directed-regLR}
and~\ref{thm:bipartite-regLR}, the latter containing the exceptions
and illustrating the surprising behaviour of the countable universal
bipartite graph).  For non-regular $\scD$\nbd classes, we observe in
Theorems~\ref{thm:graph-nonregLR}, \ref{thm:directed-nonregLR}
and~\ref{thm:bipartite-nonregLR} that there exist non-regular
injective endomorphisms with specified image and whose $\scD$\nbd
class contains both $2^{\aleph_{0}}$~many $\scL$- and $\scR$\nbd
classes.  By varying the image, we shall deduce there are
$2^{\aleph_{0}}$ non-regular $\scD$\nbd classes in the endomorphism
monoid of~$\Omega$.

A number of questions remain about the endomorphism monoid of each of
our universal structures.  For example, is it true that every
$\scD$\nbd class of the endomorphism monoid of the countable universal
graph contains $2^{\aleph_{0}}$~many $\scL$- and $\scR$\nbd classes?
This question has a positive answer for regular $\scD$\nbd classes
(Theorem~\ref{thm:graph-regLR}) and some of the non-regular $\scD$\nbd
classes (by Theorem~\ref{thm:graph-nonregLR}).  It is unclear whether
the latter can be extended to all non-regular $\scD$\nbd classes.  One
reason for the difficulty in making further progress is that we have a
necessary condition for endomorphisms to be $\scD$\nbd related in
terms of the isomorphism class of the images (in
Lemma~\ref{lem:class-facts}(iii) below) but only for regular
endomorphisms can we reverse the condition to be also sufficient.

One could also consider endomorphisms of the countable universal
linearly ordered set (that is, the rationals~$\Q$ under~$\leq$) or the
countable universal partially ordered set.  Indeed, the third author's
PhD thesis~\cite{JayThesis} contains information, including an
analogue of Theorem~\ref{thm:R-reg-classes}, about~$\End(\Q,\leq)$.
The methods are inevitably a little different and this will appear in
a subsequent publication.

\section{Preliminaries}
\label{sec:prelims}

In this section, we establish the terminology used throughout the
paper.  We summarise the basic facts about relational structures,
including what it means for them to be algebraically closed, and the
semigroup theory needed when discussing their endomorphism monoids.

\spc

A \emph{relational structure} is a pair $\Gamma = (V,\mathcal{E})$
consisting of a non-empty set~$V$ and a sequence $\mathcal{E} =
(E_{i})_{i \in I}$ of relations on~$V$.  In general, one permits
the~$E_{i}$ to have arbitrary arity, but as we are principally
concerned with (various types of) graphs it will be sufficient to deal
only with binary relations.  For convenience then we shall make this
assumption throughout.  When $\Gamma$~is a graph, we shall then also
call~$V$ the set of \emph{vertices} of~$\Gamma$.  The definitions of
graph, directed graph and bipartite graph with this viewpoint are
given at the beginning of
Sections~\ref{sec:graph}--\ref{sec:bipartite}, respectively.  A
\emph{relational substructure} of~$\Gamma$ is a relational structure
$\Delta = (U,\mathcal{D})$, where $U$~is a non-empty subset of~$V$ and
where $\mathcal{D} = (D_{i})_{i \in I}$ satisfies $D_{i} \subseteq
E_{i}$ for all~$i$.  If $U$~is a subset of~$V$, we write $\langle U
\rangle$ for the substructure~$(U,\mathcal{D})$ where $\mathcal{D} =
(D_{i})_{i \in I}$ is defined by $D_{i} = E_{i} \cap (U \times U)$ for
each~$i$.  We shall call~$\langle U \rangle$ the \emph{relational
  substructure induced by~$U$}.

If $\Gamma = (V,(E_{i})_{i \in I})$ and $\Delta = (W,(F_{i})_{i \in
  I})$ are relational structures (with relations indexed by the same
set~$I$), a \emph{homomorphism} $f \colon \Gamma \to \Delta$ is a map
$f \colon V \to W$ such that $(uf,vf) \in F_{i}$ whenever $(u,v) \in
E_{i}$.  The map $f \colon V \to W$ then induces $f \colon E_{i} \to
F_{i}$, for each $i \in I$, and we call the substructure $\im f =
(Vf,\mathcal{E}f)$, where $\mathcal{E}f = (E_{i}f)_{i \in I}$,
of~$\Delta$ the \emph{image} of~$f$.  We define the \emph{kernel}
of~$f$ to be the relation $\set{(u,v)}{uf = vf}$ on the vertex
set~$V$.  An \emph{embedding} is an injective homomorphism $f \colon
\Gamma \to \Delta$ such that, for each~$i$, \ $(u,v) \in E_{i}$ if and
only if $(uf,vf) \in F_{i}$.

\spc

In order to describe what it means for a relational structure to be
algebraically closed we shall need a little model theory.  We refer to
Hodges~\cite{Hodges} for the basic terminology.

Let $L$~be a signature and $\mathbf{K}$~be a class of $L$\nbd
structures.  A structure~$A$ in~$\mathbf{K}$ is called
\emph{algebraically closed} (in~$\mathbf{K}$) if given a
formula~$\Phi(\mathbit{x})$ of the form
\begin{equation}
  (\exists \mathbit{y}) \bigwedge_{i=1}^{k}
  \Psi_{i}(\mathbit{x},\mathbit{y}),
  \label{eq:ppf}
\end{equation}
where $k \in \N$ and each~$\Psi_{i}$ is an atomic formula, and a
finite sequence~$\mathbit{a}$ of elements of~$A$ such that there
exists an extension~$A'$ of~$A$ with $A' \models \Phi(\mathbit{a})$,
then already $A \models \Phi(\mathbit{a})$.  (As an aside, we mention
that the formula~$\Phi(\mathbit{x})$ given in~\eqref{eq:ppf} is called
a \emph{positive primitive formula}, see, for
example,~\cite[page~50]{Hodges}.)  In certain cases, this definition
of algebraic closure can often be simplified.  For example, it is easy
to see that a graph~$\Gamma$ is algebraically closed if and only if,
given any finite set~$A$ of vertices in~$\Gamma$, there exists some
vertex~$w$ that is adjacent to every one of the vertices in~$A$.  We
shall similarly interpret below what algebraically closed means for
directed and bipartite graphs in Sections~\ref{sec:directed}
and~\ref{sec:bipartite}.  The concept of an \emph{existentially
  closed} relational structure is defined similarly but for this we
permit each~$\Psi_{i}$ to be an atomic formula or its negation.

Algebraically closed structures for our classes of relational
structures can be characterised as follows.  Part~(i) of this
result is~\cite[Proposition~2.1(a)]{CamNes}, which is established
by a back-and-forth argument.  The proof is easily adjusted to cover
directed graphs and bipartite graphs, though one necessarily needs to
use the strongly algebraically closed condition for the latter.  This
condition is defined in Section~\ref{sec:bipartite} just before
Theorem~\ref{thm:ac-bipartite} where it is first used.

\begin{prop}
  \label{prop:characterise-ac}
  \begin{enumerate}
  \item Let $\Gamma = (V,E)$~be a countable graph or directed graph.
    Then $\Gamma$~is algebraically closed (in the class of graphs or
    directed graphs, respectively) if and only if there exists $F
    \subseteq E$ such that $(V,F)$~is existentially closed.
  \item Let $\Gamma = (V,E,P)$~be a countable bipartite graph.  Then
    $\Gamma$~is strongly algebraically closed if and only if there
    exists $F \subseteq E$ such that $(V,F,P)$ is existentially closed.
  \end{enumerate}
\end{prop}

The classes of finite graphs, of finite directed graphs and of finite
bipartite graphs each possess what is known as the hereditary
property, the joint embedding property and the amalgamation property.
(Indeed, the reason for our particular way of defining the term
bipartite graph below is to ensure that the class of such graphs has
these properties.)  Consequently, each class has a unique
Fra\"{\i}ss\'{e} limit~\cite{Fraisse}, referred to as the
\emph{countable universal homogeneous structure} of the class (see,
for example,~\cite[Theorem~6.1.2]{Hodges}).  We shall follow
Truss~\cite{Truss} and others and abbreviate the terminology to refer
to the \emph{countable universal graph}, the \emph{countable universal
  directed graph}, and the \emph{countable universal bipartite
  graph}.  Furthermore, these Fra\"{\i}ss\'{e} limits are the unique
countable existentially closed structures in the classes of graphs, of
directed graphs, and of bipartite graphs
(see~\cite[page~185]{Hodges}).  The following is now an immediate
corollary of Proposition~\ref{prop:characterise-ac} and is used in the
proofs of Theorems~\ref{thm:graph-nonregLR},
\ref{thm:directed-nonregLR} and~\ref{thm:bipartite-nonregLR}.

\begin{cor}\label{cor:hom-from-random}
  Let $\Gamma$~be countable and either an algebraically closed graph,
  algebraically closed directed graph, or strongly algebraically
  closed bipartite graph.  Let $\Omega$~be, correspondingly, the
  countable universal graph, countable universal directed graph, or
  countable universal bipartite graph.  Then there is a homomorphism
  from~$\Omega$ into~$\Gamma$ given by a bijection between the
  vertices. \qed
\end{cor}

Since we shall be concerned with maximal subgroups (that is, the group
$\scH$\nbd classes) of endomorphism monoids, we need to recall Green's
relations and their properties.  We refer to Howie's
monograph~\cite{Howie} for a general background on semigroups.

Let $M = \End \Gamma$ be the endomorphism monoid of a relational
structure $\Gamma = (V,\mathcal{E})$.  Two elements $f$~and~$g$ of~$M$
are \emph{$\scL$\nbd related} if $f$~and~$g$ generate the same left
ideal (that is, $Mf = Mg$), while they are \emph{$\scR$\nbd related}
if $fM = gM$.  Green's \emph{$\scH$\nbd relation} is the intersection
of the binary relations $\scL$~and~$\scR$, while the \emph{$\scD$\nbd
  relation} is their composite~$\scL \circ \scR$ (which can be shown
also to be an equivalence relation).  Finally, but less central to our
work, $f$~and~$g$ are \emph{$\scJ$\nbd related} if $MfM = MgM$.  We
shall use the notation $f \scL g$ to denote that $f$~and~$g$ are
$\scL$\nbd related and similarly for the other relations.  If $f \in
M$, we write~$H_{f}$ for the $\scH$\nbd class of~$f$.  If $e$~is an
idempotent in~$M$ (that is, $e^{2} = e$), the $\scH$\nbd class~$H_{e}$
is a subgroup of~$M$ \cite[Corollary~2.2.6]{Howie} and the maximal
subgroups of our monoid~$M$ are precisely the $\scH$\nbd classes of
idempotents of~$M$.

The $\scL$-, $\scR$- and $\scD$\nbd classes in the full transformation
monoid~$\mathcal{T}_{V}$, of all maps~$V \to V$, are fully described
in terms of the images and kernels of the maps involved
(see~\cite[Exercise~2.6.16]{Howie}).  We may view the endomorphism
monoid~$M$ of~$\Gamma = (V,\mathcal{E})$ as a submonoid
of~$\mathcal{T}_{V}$ and if $f$~and~$g$ are, for example, $\scL$\nbd
related in~$\End \Gamma$, they are certainly $\scL$\nbd related
in~$\mathcal{T}_{V}$.  Consequently, parts (i)~and~(ii) of the
following lemma follow immediately.

\begin{lemma}
  \label{lem:class-facts}
  Let $f$~and~$g$ be endomorphisms of the relational structure $\Gamma
  = (V,\mathcal{E})$.
  \begin{enumerate}
  \item If $f$~and~$g$ are $\mathscr{L}$\nbd related, then $Vf = Vg$.
  \item If $f$~and~$g$ are $\mathscr{R}$\nbd related, then $\ker f =
    \ker g$.
  \item If $f$~and~$g$ are $\mathscr{D}$\nbd related, then the induced
    substructures $\langle Vf \rangle$~and~$\langle Vg \rangle$ are
    isomorphic.
  \end{enumerate}
\end{lemma}

\begin{proof}
  (iii)~Write $\mathcal{E} = (E_{i})_{i \in I}$.  By assumption,
  there exists $h \in \End \Gamma$ such that $f \mathscr{R} h$
  and $h \mathscr{L} g$.  By~(i), it follows $Vh = Vg$.  As $f
  \mathscr{R} h$, there exist endomorphisms $s$~and~$t$ of~$\Gamma$
  with $h = fs$ and $f = ht$.  As $f = fst$ and $h = hts$, the
  map~$s$ induces a bijection from~$Vf$ to~$Vh$.  Moreover, as
  $s$~and~$t$ are endomorphisms, $s$~induces, for each $i \in I$, a
  bijection from $E_{i} \cap (Vf \times Vf)$ to $E_{i} \cap (Vh \times
  Vh)$ with inverse~$t$.  Hence $s$~induces an isomorphism
  from~$\langle Vf \rangle$ to~$\langle Vh \rangle = \langle Vg \rangle$.
\end{proof}

An element~$f$ of~$M$ is called \emph{regular} if there exists $g \in
M$ such that $fgf = f$.  An idempotent endomorphism~$e$ is regular
since $e^{3} = e$ and if $f$~is regular, then every element in the
$\scD$\nbd class of~$f$ is also
regular~\cite[Proposition~2.3.1]{Howie}.  We refer to such $\scD$\nbd
classes as \emph{regular $\scD$\nbd classes}.  We are particularly
concerned with idempotent endomorphisms and their $\scH$- and
$\scD$\nbd classes and so the observation in
Lemma~\ref{lem:reg-classes} below that the implications in
Lemma~\ref{lem:class-facts} reverse for regular elements is useful.

If $f$~is any endomorphism of $\Gamma = (V,\mathcal{E})$, where
$\mathcal{E} = (E_{i})_{i \in I}$, then immediately $E_{i}f \subseteq
E_{i} \cap (Vf \times Vf)$ for all~$i$.  On the other hand, if $f$~is
regular, say $fgf = f$ for $g \in M$, then $gf$~is idempotent and it
is easy to check that $Vgf = Vf$.   Hence if $(x,y) \in E_{i} \cap (Vf
\times Vf)$, then $x,y \in Vgf$ and $(x,y) = (x,y)gf \in E_{i}f$.
Consequently $E_{i} \cap (Vf \times Vf) = E_{i}f$, which establishes
that for regular endomorphisms our two possible definitions of image
coincide.

\begin{prop}
  \label{prop:reg-image}
  Let $f$~be a regular endomorphism of the relational structure
  $\Gamma = (V,\mathcal{E})$.  Then the image $\im f = (Vf,
  \mathcal{E}f)$ and the induced substructue~$\langle Vf \rangle$
  of\/~$\Gamma$ are equal. \qed
\end{prop}

\begin{lemma}
  \label{lem:reg-classes}
  Let $f$~and~$g$ be regular elements in the endomorphism monoid of
  the relational structure $\Gamma = (V,\mathcal{E})$.  Then
  \begin{enumerate}
  \item $f$~and~$g$ are $\mathscr{L}$\nbd related if and only if\/ $Vf
    = Vg$;
  \item $f$~and~$g$ are $\mathscr{R}$\nbd related if and only if\/
    $\ker f = \ker g$;
  \item $f$~and~$g$ are $\mathscr{D}$\nbd related if and only if the
    images of $f$~and~$g$ are isomorphic.
  \end{enumerate}
\end{lemma}

\begin{proof}
  It suffices to establish the ``if'' versions of each part.  For
  (i)~and~(ii) this follows immediately since, for example, if $Vf =
  Vg$, then $f$~and~$g$ are $\mathscr{L}$\nbd related
  in~$\mathcal{T}_{V}$ (by \cite[Exercise~2.6.16]{Howie}) and hence
  are $\mathscr{L}$\nbd related in~$\End \Gamma$
  (see~\cite[Proposition~A.1.16]{RhoSte}).

  (iii)~A more general version of this result is Theorem~2.6
  in~\cite{MagSub}, but what we require can be completed easily.  If
  $\alpha$~is an isomorphism from~$\im f$ to~$\im g$, then
  $f\alpha$~is an endomorphism of~$\Gamma$ with image~$\im g$.  As
  $g$~is regular, there is an idempotent endomorphism~$e$ that is
  $\mathscr{L}$\nbd related to~$g$,
  by~\cite[Proposition~2.3.2]{Howie}.  By~(i), $Ve = Vg$ and therefore
  $\im g = \im e$, using Proposition~\ref{prop:reg-image}.  Hence the
  restriction $e|_{\im g}$~is the identity.  Then $e\alpha^{-1}$~is an
  endomorphism of~$\Gamma$ such that $f\alpha \cdot e\alpha^{-1} =
  f\alpha\alpha^{-1} = f$ and we conclude that $f \mathscr{R}
  f\alpha$.  It follows, by~\cite[Proposition~2.3.1]{Howie}, that
  $f\alpha$~is also a regular element and so $f\!\alpha\, \mathscr{L}
  g$ by~(i).  Hence $f \mathscr{D} g$, as required.
\end{proof}

One might ask what happens in the case that the $\scH$\nbd class $H =
H_{f}$, of some endomorphism~$f$ in $M = \End\Gamma$, is not a group.
In such a case, one can associate to~$H$ the \emph{\Schutz\ group},
which we shall denote~$\mathcal{S}_{H}$.  This consists of the
permutations of the $\scH$\nbd class~$H$ induced by certain elements
of~$M$.  Specifically, we define $T_{H} = \set{t \in M}{Ht \subseteq
  H}$ and $\mathcal{S}_{H} = \set{\gamma_{t}}{t \in T_{H}}$, where the
map $\gamma_{t} \colon H \to H$ is given by $h \mapsto ht$.  It is
known (see, for example, \cite[Theorem~2.22]{CP1}) that
$\mathcal{S}_{H}$~is a group and if $H$~is itself a group (for
example, when $f$~is an idempotent) then $\mathcal{S}_{H} \cong H$.
Moreover, two $\scH$\nbd classes in the same $\scD$\nbd class have
isomorphic \Schutz\ groups (see \cite[Theorem~2.25]{CP1}), which is
why we produce distinct $\scD$\nbd classes in
Theorems~\ref{thm:R-uncountSchutz}, \ref{thm:D-uncountSchutz}
and~\ref{thm:B-uncountSchutz}.  Amongst other things, we shall observe
that in our context this group can be expressed as a subgroup of the
automorphism group of the image of~$f$.

\begin{prop}
  \label{prop:Schutzgps}
  Let $f$~be an endomorphism of the relational structure $\Gamma =
  (V,\mathcal{E})$ and $H$~be the $\scH$\nbd class of~$f$ in~$\End
  \Gamma$.  Then
  \begin{enumerate}
  \item if $t \in T_{H}$, the restriction of~$t$ to the set~$Vf$
    induces an automorphism of both~$\langle Vf \rangle$ and\/~$\im
    f$;
  \item the mapping $\phi \colon \gamma_{t} \mapsto t|_{Vf}$ (for $t
    \in T_{H}$) defines an injective homomorphism from the
    \Schutz\ group~$\mathcal{S}_{H}$ into $\Aut \langle Vf \rangle
    \cap \Aut ( \im f )$;
  \item if $f$~is an idempotent endomorphism, then $H \cong \Aut(\im
    f)$ (as groups);
  \item if $f \colon V \to V$~is injective and defines an endomorphism
    of\/~$\Gamma$ and $g$~is an automorphism of\/~$\im f$, then
    $fg$~is $\scL$\nbd related to~$f$.
  \end{enumerate}
\end{prop}

Part~(iii) of this lemma can also be shown directly without
reference to the \Schutz\ group; see \cite[Theorem~2.7]{JayThesis}.
When $\Gamma$~is one of the existentially closed structures that we
are interested in and $f$~arises in a specific way, we shall extend
part~(iv) to show that $fg$~is actually $\scH$\nbd related to~$f$ and
hence that the image of~$\phi$ in part~(ii) is $\Aut \langle Vf
\rangle \cap \Aut(\im f)$ \ (see Propositions~\ref{prop:R-Schutz},
\ref{prop:D-Schutz} and~\ref{prop:B-Schutz} below).

\begin{proof}
  (i)~If $t \in T_{H}$, then $ft$~is $\scH$\nbd related to~$f$ and so
  $Vft = Vf$ and $\ker ft = \ker f$, by Lemma~\ref{lem:class-facts}.
  It follows that $t$~induces a bijection on the set~$Vf$ and hence an
  endomorphism of the substructure~$\langle Vf \rangle$.  Also there
  exists some endomorphism~$s$ such that $fts = f$.  Hence if
  $(uft,vft) \in E_{i}$ for some relation~$E_{i} \in \mathcal{E}$,
  then so is $(uf,vf) = (ufts,vfts)$ and we conclude that $t$~induces
  an automorphism of~$\langle Vf \rangle$.

  Now $ft$~is, in particular, $\scL$\nbd related to~$f$ and so there
  exist endomorphisms $g$~and~$h$ of~$\Gamma$ such that $ft = gf$ and
  $hft = f$.  Let $(v_{1},v_{2}) \in E_{i}f$ for some $E_{i} \in
  \mathcal{E}$, so $v_{j} = u_{j}f$ for some points~$u_{j} \in V$,
  for~$j = 1$,~$2$, with $(u_{1},u_{2}) \in E_{i}$.  Then $v_{j}t =
  u_{j}ft = u_{j}gf$ and we conclude $(v_{1}t,v_{2}t)$~is the image
  of~$(u_{1},u_{2})$ under the endomorphism~$gf$ and so
  $(v_{1}t,v_{2}t) \in E_{i}f$.  Thus $t$~induces an endomorphism
  of~$\im f$.  In addition, the endomorphism~$s$ in the previous
  paragraph is the inverse of~$t$ on the set~$Vf$ and, since $fs =
  hfts = hf$, we similarly conclude $s$~induces an endomorphism
  of~$\im f$.  Hence $t$~induces an automorphism of the image.

  (ii)~Note that if $\gamma_{s} = \gamma_{t}$ for some~$s,t \in
  T_{H}$, then in particular $fs = ft$ and so the restrictions of
  $s$~and~$t$ to~$Vf$ coincide.  Conversely, if these restrictions
  coincide then $fs = ft$ and we conclude that $hs = ht$ for all~$h$
  that are $\mathscr{L}$\nbd related to~$f$.  Hence $\gamma_{s} =
  \gamma_{t}$.  Therefore, using~(i), we observe that $\phi$~is a
  injective map and it is straightforward to see that it is also a
  homomorphism.

  (iii)~Since $f$~is an idempotent endomorphism, $\langle Vf \rangle =
  \im f$, by Proposition~\ref{prop:reg-image}.  Then, given an
  automorphism~$g$ of~$\im f$, note that $fg$~is $\scH$\nbd related
  to~$f$, because $(fg)(fg^{-1}) = f$ since $f$~acts as the identity
  on its image.  We then see that $\gamma_{fg}\phi$~has the same
  effect on points in~$Vf$ as $g$~does.  Hence $\phi \colon
  \mathcal{S} \to \Aut ( \im f )$~is surjective, as required to
  establish the isomorphism.

  (iv)~If $g$~is an automorphism of~$\im f$, define $h \colon V \to V$
  by setting~$vh$ to be the unique point satisfying $vhf = vfg$.
  Since $f$~is injective and $g$~is a bijection on the set~$Vf$, we
  conclude that $h$~is a permutation of~$V$.  As both $g$~and~$g^{-1}$
  are automorphisms of~$\im f$, we deduce that $h$~is an automorphism
  of~$\Gamma$.  Then, from $f = h^{-1}fg$, we conclude that
  $fg$~and~$f$ are $\scL$\nbd related.
\end{proof}

\section{Graphs}
\label{sec:graph}

In this paper, a \emph{graph} will have its usual definition; that is,
a relational structure $\Gamma = (V,E)$ where $V$~is the set of
vertices and $E$~is an irreflexive symmetric binary relation on~$V$.
Thus the term graph refers to an undirected graph without loops or
multiple edges.  If $(u,v)$~is an edge in~$E$, we then say that the
vertices $u$~and~$v$ are \emph{adjacent} in~$\Gamma$.  Recall that a
graph~$\Gamma$ is \emph{algebraically closed} if for every finite
subset~$A$ of its vertices, there exists some vertex~$v$ such that
$v$~is adjacent to every member of~$A$.

In order to establish our results, we introduce a number of
constructions.  If $\Gamma = (V,E)$ is any graph, we define the
\emph{complement} of~$\Gamma$ to be the graph~$\Gamma^{\dagger}$ with
vertex set~$V$ and edge set $( V \times V ) \setminus (E \cup \set{
  (v,v) }{ v \in V } )$.  Thus $\Gamma^{\dagger}$~is the graph
containing precisely all the edges that are not present in~$\Gamma$.
We observe immediately:

\begin{lemma}
  \label{lem:complement}
  Let $\Gamma$~and~$\Delta$ be any graphs.  Then
  {\normalfont(i)}~$\Aut \Gamma^{\dagger} = \Aut \Gamma$;
  {\normalfont(ii)}~$\Gamma \cong \Delta$ if and only if\/
  $\Gamma^{\dagger} \cong \Delta^{\dagger}$. \qed
\end{lemma}

Recall that a graph is \emph{locally finite} if every vertex is
adjacent to a finite number of vertices.  If $\Gamma = (V,E)$ and
$\Delta = (W,F)$ are two graphs (where the vertex sets $V$~and~$W$ are
assumed disjoint), then the \emph{disjoint union}~$\Gamma
\disjunion \Delta$ is the graph with vertex set~$V \cup W$ and edge
set~$E \cup F$.  These two concepts may be used to construct an
algebraically closed graph as follows:

\begin{lemma}
  \label{lem:ac-construct}
  Let $\Gamma$~be any graph and $\Lambda$~be an infinite locally
  finite graph.  Then $( \Gamma \disjunion \Lambda)^{\dagger}$~is
  algebraically closed.
\end{lemma}

\begin{proof}
  Let $V$~and~$W$ denote the vertex sets of $\Gamma$~and~$\Lambda$
  respectively and let $\Delta = (\Gamma \disjunion
  \Lambda)^{\dagger}$.  If $A$~is a finite subset of~$V \cup W$ then,
  since $\Lambda$~is locally finite, there exists some vertex~$v \in
  W$ that is not adjacent in~$\Lambda$ to any vertex in~$A \cap W$.
  Consequently, $v$~is not adjacent in~$\Gamma \disjunion \Lambda$ to
  any vertex in~$A$ and so by construction $v$~is adjacent in~$\Delta$
  to every vertex in~$A$.
\end{proof}

We shall use here, and also in later sections, the locally finite
graphs~$L_{S}$ defined as follows.  Let $S$~be any subset of~$\N
\setminus \{0,1\}$.  The set of vertices of~$L_{S}$ is $\set{ \ell_{n}
}{n \in \N} \cup \set{ v_{n} }{ n \in S }$.  For every $n \in \N$,
vertex~$\ell_{n}$ is adjacent to~$\ell_{n+1}$, while for every $n \in
S$, vertex~$\ell_{n}$ is also adjacent to~$v_{n}$.  See
Figure~\ref{fig:L-graph} for a diagram of an example of~$L_{S}$.  The
following presents the basic information we need about the
graphs~$L_{S}$.  If $S \subseteq \N \setminus \{ 0,1 \}$, then we
write~$S + k$ for the set $\set{n+k}{n \in S}$.

\begin{figure}
  \begin{center}
    \begin{pspicture}(0,.5)(7,1)
      \multips(0,0)(1,0){7}{\qdisk(0,0){.05}}
      \qdisk(2,1){.05}
      \qdisk(4,1){.05}
      \qdisk(5,1){.05}
      \put(6.3,0){\dots}
      \psline(0,0)(6.2,0)
      \psline(2,0)(2,1)
      \psline(4,0)(4,1)
      \psline(5,0)(5,1)
    \end{pspicture}
  \end{center}
  \caption{The graph~$L_{\{2,4,5,\dots\}}$}
  \label{fig:L-graph}
\end{figure}
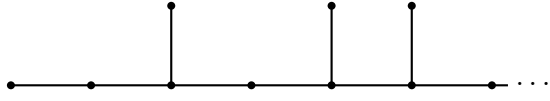

\begin{lemma}
  \label{lem:L-props}
  Let $S,T \subseteq \N \setminus \{0,1\}$.  Then
  \begin{enumerate}
  \item $\Aut L_{S} = \1$;
  \item there exists a graph homomorphism $f \colon L_{S} \to L_{T}$
    defined by an injective map on the sets of vertices if and only if
    there exists some $k \in \N$ such that $S + k \subseteq T$;
  \item $L_{S} \cong L_{T}$ if and only if $S = T$.
  \end{enumerate}
\end{lemma}

\begin{proof}
  (i)~By construction, $\ell_{0}$~is the only vertex of degree~$1$
  in~$L_{S}$ that is adjacent to a vertex of degree~$2$.  All other
  vertices~$\ell_{n}$ (for $n \geq 1$) have degree at least~$2$.  All
  vertices~$v_{n}$ (for $n \in S$) have degree~$1$ and are adjacent to
  vertices~$\ell_{n}$ of degree~$3$.  It follows that $\Aut L_{S} =
  \1$.

  (ii)~Suppose $f \colon L_{S} \to L_{T}$ is a graph homomorphism
  given by an injective map on the sets of vertices.  Then $f$~must
  map the infinite path $\{ (\ell_{0},\ell_{1}), (\ell_{1},\ell_{2}),
  \dots \}$ in~$L_{S}$ to an infinite path of distinct vertices
  in~$L_{T}$.  Hence there exists $k \in \N$ such that $\ell_{n}f =
  \ell_{n+k}$ for all $n \in \N$.  In order that edges of the
  form~$(\ell_{n},v_{n})$ in~$L_{S}$ are mapped to edges in~$L_{T}$,
  it follows that $S + k \subseteq T$.

  Conversely, if $S + k \subseteq T$, then the map $f \colon L_{S} \to
  L_{T}$ given by $\ell_{n}f = \ell_{n+k}$ for $n \in \N$ and $v_{n}f
  = v_{n+k}$ for $n \in S$ is a graph homomorphism.

  (iii)~follows immediately from~(ii).
\end{proof}

We now establish the first of our main theorems for graphs.

\begin{thm}
  \label{thm:ac-graphs}
  Let $\Gamma$~be a countable graph.  Then there exist
  $2^{\aleph_{0}}$~pairwise non-isomorphic countable algebraically
  closed graphs whose automorphism group is isomorphic to that
  of\/~$\Gamma$.
\end{thm}

\begin{proof}
  Fix the countable graph~$\Gamma$.  This has at most countably many
  connected components and so, by Lemma~\ref{lem:L-props}(iii),
  there are $2^{\aleph_{0}}$~choices of subsets~$S$ of~$\N \setminus
  \{ 0,1 \}$ such that $L_{S}$~is isomorphic to no component
  of~$\Gamma$.  For such a choice of~$S$,
  \[
  \Aut (\Gamma \disjunion L_{S})^{\dagger} = \Aut (\Gamma \disjunion
  L_{S}) \cong \Aut \Gamma \times \Aut L_{S} \cong \Aut \Gamma.
  \]
  Hence there are $2^{\aleph_{0}}$~choices of~$S$ such that
  $\Delta_{S} = (\Gamma \disjunion L_{S})^{\dagger}$ has automorphism
  group isomorphic to that of~$\Gamma$.  Lemma~\ref{lem:ac-construct}
  tells us that each~$\Delta_{S}$ is algebraically closed.

  Finally, if $S$~and~$T$ are distinct subsets of~$\N \setminus
  \{0,1\}$ such that neither $L_{S}$~nor~$L_{T}$ are isomorphic to a
  connected component of~$\Gamma$, then $L_{S} \not\cong L_{T}$ by
  Lemma~\ref{lem:L-props}(iii).  It then follows that $\Gamma
  \disjunion L_{S} \not\cong \Gamma \disjunion L_{T}$ and hence
  $\Delta_{S} \not\cong \Delta_{T}$ by
  Lemma~\ref{lem:complement}(ii).  This completes the proof.
\end{proof}

To establish the required information about images of endomorphisms of
the countable universal graph~$R$, we shall remind the reader of a
standard way to construct~$R$ by building it around any countable
graph.  For a countable graph $\Gamma = (V,E)$, construct a new
graph~$\mathcal{G}(\Gamma)$ as follows.  Enumerate the finite subsets
of~$V$ as~$(A_{i})_{i \in I}$ where $I \subseteq \N$.  For each~$i \in
I$, let $v_{i}$~be a new vertex.  Define~$\mathcal{G}(\Gamma)$ to be
the graph with vertex set $V \cup \set{v_{i}}{i \in I}$ and edge set
\[
E \cup \set{ (v_{i},a), (a,v_{i}) }{ a \in A_{i}, \, i \in I };
\]
thus, we have, for each~$i$, added a new vertex~$v_{i}$ that is
adjacent to every vertex in~$A_{i}$ but to no other vertex
in~$\mathcal{G}(\Gamma)$.  Now construct a sequence of
graphs~$\Gamma_{n}$ by defining $\Gamma_{0} = \Gamma$ and
$\Gamma_{n+1} = \mathcal{G}(\Gamma_{n})$ for each~$n \geq 0$.  Since
each~$\Gamma_{n}$ is naturally a subgraph of~$\Gamma_{n+1}$, we can
define~$\Gamma_{\infty}$ to be the limit of this sequence of graphs.
The resulting graph is countable and is, by construction,
existentially closed and so isomorphic to the countable universal
graph~$R$.

Suppose now that we also have a graph homomorphism $f \colon \Gamma
\to \Gamma_{\infty}$.  As before, enumerate the finite subsets of~$V$
as~$(A_{i})_{i \in \N}$.  We shall define an extension $\tilde{f}
\colon \mathcal{G}(\Gamma) \to \Gamma_{\infty}$.  Indeed, suppose that
a graph homomorphism~$f_{n}$ has been defined with domain equal to the
subgraph of~$\mathcal{G}(\Gamma)$ induced by~$V \cup \{
v_{1},v_{2},\dots,v_{n} \}$ and such that the restriction of~$f_{n}$
to~$\Gamma$ equals~$f$.  As $\Gamma_{\infty}$~is in particular
algebraically closed, there exists a vertex~$w$ within it that is
adjacent to every vertex in $(A_{n+1} \cup \{ v_{1},v_{2},\dots,v_{n}
\} )f_{n}$.  We extend to a function~$f_{n+1}$ with domain equal to
the subgraph of~$\mathcal{G}(\Gamma)$ induced by $V \cup
\{v_{1},v_{2},\dots,v_{n+1} \}$ by defining $v_{n+1}f_{n+1} = w$.  The
choice of~$w$ ensures that $f_{n+1}$~is a graph homomorphism.  Note
that, if $\im f$~was originally an algebraically closed subgraph
of~$\Gamma_{\infty}$, then we could at every stage choose $w_{n+1} \in
\im f$.  Consequently, in this case we can arrange for $\im f_{n} =
\im f$ for all~$n$.

Since each~$f_{n+1}$ extends~$f_{n}$, we may define $\tilde{f} =
\lim_{n\to\infty} f_{n} = \bigcup_{n=0}^{\infty} f_{n}$.  Then
$\tilde{f}$~is a graph homomorphism $\mathcal{G}(\Gamma) \to
\Gamma_{\infty}$ whose restriction to~$\Gamma$ equals~$f$.  Moreover,
if $\im f$~is algebraically closed, we can arrange that $\im \tilde{f}
= \im f$.  We use this contruction to establish the first two parts of
the following result (a variant of which appears in a more general
setting as~\cite[Theorem~4.1]{Dolinka2}):

\begin{lemma}
  \label{lem:map-extend}
  Let $\Gamma$~be a countable graph, let $\Gamma_{\infty}$~be the copy
  of the countable universal graph constructed around\/~$\Gamma$ as
  described above, and let $f \colon \Gamma \to \Gamma_{\infty}$ be a
  graph homomorphism.
  \begin{enumerate}
  \item There exist $2^{\aleph_{0}}$~endomorphisms $\hat{f} \colon
    \Gamma_{\infty} \to \Gamma_{\infty}$ such that the restriction
    of~$\hat{f}$ to~$\Gamma$ equals~$f$.
  \item If\/ $\im f$~is algebraically closed, then there are
    $2^{\aleph_{0}}$~such extensions~$\hat{f}$ of~$f$ with $\im
    \hat{f} = \im f$.
  \item If $f$~is an automorphism of\/~$\Gamma$, then there is an
    automorphism~$\hat{f}$ of\/~$\Gamma_{\infty}$ such that the
    restriction of~$\hat{f}$ to~$\Gamma$ equals~$f$.
  \end{enumerate}
\end{lemma}

\begin{proof}
  (i),~(ii): Our recipe above describes how to extend a graph
  homomorphism $f \colon \Gamma \to \Gamma_{\infty}$ to $\tilde{f}
  \colon \mathcal{G}(\Gamma) \to \Gamma_{\infty}$.  As
  $\Gamma_{\infty}$~is defined as the limit of the sequence given by
  $\Gamma_{0} = \Gamma$, \ $\Gamma_{n+1} = \mathcal{G}(\Gamma_{n})$,
  repeated use of this recipe constructs one example of the required
  endomorphism~$\hat{f}$.  Now observe that algebraic closure ensures
  there are infinitely many suitable vertices~$w$ adjacent to every
  vertex in~$(A_{n+1} \cup \{ v_{1},v_{2},\dots,v_{n} \} )f_{n}$.
  This freedom at each stage ensures that there are uncountably many
  possible extensions~$\hat{f}$.

  (iii)~This is achieved by a variant construction.  The
  automorphism~$f$ induces a permutation of the finite
  subsets~$A_{i}$.  If $A_{i}f = A_{j}$, then define the extension
  $\tilde{f} \colon \mathcal{G}(\Gamma) \to \mathcal{G}(\Gamma)$ by
  setting $v_{i}\tilde{f} = v_{j}$.  This is an automorphism
  of~$\mathcal{G}(\Gamma)$ and repeated use of this construction
  yields the required extension to~$\Gamma_{\infty}$.
\end{proof}

We can now observe that images of idempotent endomorphisms of
the countable universal graph~$R$ are characterized by being
algebraically closed.  This was established by Bonato and
Deli\'{c}~\cite[Proposition~4.2]{BonDel} and is part~(i) of the
following theorem.  However, our proof shows there are in fact
uncountably many idempotents with specified (algebraically closed)
image.

\begin{thm}
  \label{thm:graph-idemp}
  Let $\Gamma$~be a countable graph.  Then
  \begin{enumerate}
  \item there exists an idempotent endomorphism~$f$ of the countable
    universal graph~$R$ such that $\im f \cong \Gamma$ if and only
    if\/ $\Gamma$~is algebraically closed;
  \item if\/ $\Gamma$~is algebraically closed, there are
    $2^{\aleph_{0}}$~idempotent endomorphisms~$f$ of~$R$ such that
    $\im f \cong \Gamma$.
  \end{enumerate}
\end{thm}

\begin{proof}
  An existentially closed graph is certainly also algebraically
  closed and it is easy to see that this latter property is inherited
  by images of an endomorphism~$f$.

  Conversely, if $\Gamma$~is algebraically closed, let $f$~be one of
  the extensions to~$\Gamma_{\infty}$, given by
  Lemma~\ref{lem:map-extend}(ii), of the identity map $\Gamma \to
  \Gamma$.  We identify~$\Gamma_{\infty}$ with~$R$.  Then the
  restriction of~$f$ to $\im f = \Gamma$ is the identity and so $f$~is
  an idempotent endomorphism of~$R$ with image isomorphic to~$\Gamma$.
  Moreover, as observed in Lemma~\ref{lem:map-extend}, there are
  actually $2^{\aleph_{0}}$~many such idempotent endomorphisms.
  Consequently, part~(ii) also follows.
\end{proof}

We may now establish that any suitable group arises in $2^{\aleph_{0}}$~ways
as a maximal subgroup of~$\End R$.

\begin{thm}
  \label{thm:R-reg-classes}
  Let $R$~denote the countable universal graph.
  \begin{enumerate}
  \item Let\/ $\Gamma$~be a countable graph.  Then there exist
    $2^{\aleph_{0}}$~distinct regular $\scD$\nbd classes of\/~$\End R$
    whose group $\scH$\nbd classes are isomorphic to~$\Aut \Gamma$.
  \item Every regular $\scD$\nbd class of\/~$\End R$ contains
    $2^{\aleph_{0}}$~distinct group $\scH$\nbd classes.
  \end{enumerate}
\end{thm}

\begin{proof}
  (i)~By Theorem~\ref{thm:ac-graphs}, there are
  $2^{\aleph_{0}}$~pairwise non-isomorphic countable algebraically
  closed graphs with automorphism group isomorphic to that
  of~$\Gamma$.  For each such graph~$\Delta$, there is an idempotent
  endomorphism~$f_{\Delta}$ of~$R$ with $\im f_{\Delta} \cong \Delta$
  by Theorem~\ref{thm:graph-idemp}(i).  The idempotents~$f_{\Delta}$
  belong to distinct $\scD$\nbd classes, by
  Lemma~\ref{lem:reg-classes}(iii).  By
  Proposition~\ref{prop:Schutzgps}(iii), the corresponding group
  $\scH$\nbd class satisfies $H_{f_{\Delta}} \cong \Aut \Delta \cong
  \Aut \Gamma$.  This establishes part~(i).

  (ii)~Let $D_{f}$~be a regular $\scD$\nbd class in~$\End R$ with
  $f$~an idempotent endomorphism belonging to this class.  Let $\Gamma
  = \im f$.  Then by Theorem~\ref{thm:graph-idemp}, there exist
  $2^{\aleph_{0}}$~idempotent endomorphisms of~$R$ with image
  isomorphic to~$\Gamma$.  Each such endomorphism is $\scD$\nbd
  related to~$f$ by Lemma~\ref{lem:reg-classes}(iii) but lies in a
  distinct $\scH$\nbd class by parts (i)~and~(ii) of that lemma.
\end{proof}

We now turn to the $\scL$- and $\scR$\nbd classes in the endomorphism
monoid of~$R$.  We use the graph~$\Gamma^{\sharp}$ constructed from a
countably infinite graph~$\Gamma$ by, loosely speaking, replacing
every edge in~$\Gamma$ by a copy of the complete bipartite
graph~$K_{2,2}$.  More precisely, if $\Gamma = (V,E)$ and $V =
\set{v_{i}}{i \in \N}$, then define $\Gamma^{\sharp} =
(V^{\sharp},E^{\sharp})$ where
\[
V^{\sharp} = \set{ v_{i,r} }{ i \in \N, \; r\in\{0,1\} }
\]
and
\[
E^{\sharp} = \set{ (v_{i,r},v_{j,s}) }{ (v_i,v_j)\in E, \; r,s \in
  \{0,1\}}.
\]
The following observations are straightforward.

\begin{lemma}
  \label{lem:sharp}
  Let\/ $\Gamma$~be any countably infinite graph.
  \begin{enumerate}
  \item If\/ $\Gamma$~is algebraically closed, then so
    is~$\Gamma^{\sharp}$.
  \item For any sequence $(b_{i})_{i\in\N}$ with $b_{i} \in \{0,1\}$
    for all~$i$, the subgraph of $\Gamma^{\sharp}$ induced by the
    vertices $\set{v_{i,b_{i}}}{i\in\N}$ is isomorphic to~$\Gamma$. \qed
  \end{enumerate}
\end{lemma}

\begin{thm}
  \label{thm:graph-regLR}
  Every regular $\scD$\nbd class of the endomorphism monoid of the
  countable universal graph~$R$ contains $2^{\aleph_{0}}$~many
  $\scL$- and $\scR$\nbd classes.
\end{thm}

\begin{proof}
  Fix an idempotent endomorphism~$f$ of~$R$ and let $\Gamma$~be the
  image of~$f$, which is algebraically closed by
  Theorem~\ref{thm:graph-idemp}.  First assume that $R$~is constructed
  as~$\Gamma_{\infty}$ as described above by taking $\Gamma_{0} =
  \Gamma$.  Lemma~\ref{lem:map-extend}(ii) tells us that there
  $2^{\aleph_{0}}$~extensions to~$R$ of the identity map on~$\Gamma$
  with the same image and all such extensions of $\scD$\nbd related
  to~$f$ by Lemma~\ref{lem:reg-classes}(iii).  However, as idempotents
  with the same image, they have distinct kernels and so are not
  $\scR$\nbd related.

  On the other hand, we may start with the same graph~$\Gamma$,
  form~$\Gamma^{\sharp}$ as described above and then
  construct~$\Gamma_{\infty} \cong R$ now taking $\Gamma_{0} =
  \Gamma^{\sharp}$.  As $\Gamma^{\sharp}$~is also algebraically
  closed, we can extend the identity map on~$\Gamma^{\sharp}$ to an
  idempotent endomorphism~$g$ of~$R$ with image equal
  to~$\Gamma^{\sharp}$.  Now let $\mathbf{b} = (b_{i})_{i \in I}$ be
  an arbitrary sequence with $b_{i} \in \{ 0,1 \}$ for each~$i$ and
  define $\phi_{\mathbf{b}} \colon \Gamma^{\sharp} \to
  \Gamma^{\sharp}$ by
  \[
  v_{i,r}\phi_{\mathbf{b}} = v_{i,b_{i}}.
  \]
  Then $\phi_{\mathbf{b}}$~is an endomorphism of~$\Gamma^{\sharp}$
  with image equal to the subgraph~$\Lambda_{\mathbf{b}} \cong \Gamma$
  induced by the set of vertices $\set{v_{i,b_{i}}}{i \in \N}$.  Note
  that~$g\phi_{\mathbf{b}}$ is an idempotent endomorphism of~$R$ with
  image equal to~$\Lambda_{\mathbf{b}}$ and hence is $\scD$\nbd
  related to~$f$ by Lemma~\ref{lem:reg-classes}(iii).  As we
  permit~$\mathbf{b}$ to vary, we produce endomorphisms that are not
  $\scL$\nbd related, by Lemma~\ref{lem:reg-classes}(i), since if
  $\mathbf{b} \neq \mathbf{c}$ then $\Lambda_{\mathbf{b}} \neq
  \Lambda_{\mathbf{c}}$.

  This establishes that the $\scD$\nbd class of~$f$ has both
  $2^{\aleph_{0}}$~many $\scL$- and $\scR$\nbd classes within it.
\end{proof}

\begin{comment}
  It is possible to construct a collection~$\mathcal{P}$ of
  $2^{\aleph_{0}}$~subsets of~$\N$ such that for all distinct pairs
  $S,T \in \mathcal{P}$ and all positive integers~$k$, the
  translate~$S+k$ is not contained in~$T$.  Consequently, the
  graph~$L_{S}$ cannot be embedded in any~$L_{T}$ for $S,T \in
  \mathcal{P}$.  Let us write $f_{S}$~for an idempotent with image
  isomorphic to the graph~$\Delta_{S}$ as defined in the proof of
  Theorem~\ref{thm:ac-graphs}.  It then follows that $\im
  f_{S}$~cannot be embedded in~$\im f_{T}$ for distinct $S,T \in
  \mathcal{P}$ and this is sufficient to establish that $\End R$~has
  $2^{\aleph_{0}}$ many $\mathscr{J}$\nbd classes.
  See~\cite[Theorem~3.32]{JayThesis} for more details.
\end{comment}

\begin{thm}
  \label{thm:graph-nonregLR}
  Let $\Gamma$~be any countable algebraically closed graph that is not
  isomorphic to the countable universal graph~$R$.  Then there exists
  a non-regular injective endomorphism~$f$ of~$R$ such that the
  subgraph induced by the images of the vertices under~$f$ is
  isomorphic to~$\Gamma$ and such that the $\scD$\nbd class of~$f$
  contains $2^{\aleph_{0}}$ many $\mathscr{R}$- and $\mathscr{L}$\nbd
  classes.
\end{thm}

\begin{proof}
  From the graph~$\Gamma$, first build~$\Gamma^{\sharp}$ as above and
  take $\Gamma_{0} = \Gamma^{\sharp}$ when building~$\Gamma_{\infty}
  \cong R$ as described earlier.  We may thus assume that the
  countable universal graph~$R$ contains amongst its vertices
  the~$v_{i,r}$.  Write $V$~for the set of vertices of~$R$.  Let
  $\Lambda_{0}$~denote the subgraph of~$\Gamma^{\sharp}$ induced by
  the set of vertices $V_{0} = \set{ v_{i,0} }{i \in \N}$.  Then, by
  Lemma~\ref{lem:sharp}(ii), $\Lambda_{0} \cong \Gamma \not\cong R$.
  There is therefore, by Corollary~\ref{cor:hom-from-random}, a
  bijection $V \to V_{0}$ defining a graph homomorphism~$f \colon R
  \to \Lambda_{0}$.  By construction, $\langle Vf \rangle = \langle
  V_{0} \rangle \cong \Gamma$.  We shall view~$f$ as an endomorphism
  of~$R$ via the constructed embedding of~$\Lambda_{0}$ in~$R$.  Since
  $\Lambda_{0} \not\cong R$, there must exist a pair of vertices
  $u$~and~$v$ in~$R$ that are not adjacent but such that $(uf,vf)$~is
  an edge in~$\Lambda_{0}$.  Consequently, $f$~is not regular by
  Proposition~\ref{prop:reg-image}.

  A variant of the argument used in Theorem~\ref{thm:graph-regLR}
  shows that the $\mathscr{R}$\nbd class of~$f$ contains
  $2^{\aleph_{0}}$ many $\scH$\nbd classes, i.e., that it intersects
  $2^{\aleph_{0}}$ $\mathscr{L}$\nbd classes in the $\scD$\nbd class
  of~$f$.  Let $\mathbf{b}=(b_{i})_{i\in\N}$ be an arbitrary sequence
  with $b_{i} \in \{0,1\}$ for each~$i$ and define
  $\psi_{\mathbf{b}} \colon \Gamma^{\sharp}\to \Gamma^{\sharp}$ by
  \[
  v_{i,j}\psi_{\mathbf{b}} = v_{i,j+b_{i}}
  \]
  (where, in the subscript, we perform addition in~$\{0,1\}$
  modulo~$2$).  It follows from the definition of~$\Gamma^{\sharp}$
  that $\psi_{\mathbf{b}}$~is an automorphism of this graph.  By
  Lemma~\ref{lem:map-extend}(iii), $\psi_{\mathbf{b}}$~can be
  extended to an automorphism~$\hat{\psi}_{\mathbf{b}}$ of~$R$.
  Certainly $f\hat{\psi}_{\mathbf{b}}$~is $\mathscr{R}$\nbd related
  to~$f$.  Now $Vf\hat{\psi}_{\mathbf{b}} = V_{0}\psi_{\mathbf{b}}$
  and so, using Lemma~\ref{lem:class-facts}(i),
  $f\hat{\psi}_{\mathbf{b}}$~and~$f\hat{\psi}_{\mathbf{c}}$ are not
  $\mathscr{L}$\nbd related if $\mathbf{b}$~and~$\mathbf{c}$ are
  different sequences, since $V_{0}\psi_{\mathbf{b}} \neq
  V_{0}\psi_{\mathbf{c}}$.  It follows that the $\mathscr{R}$-class
  of~$f$ contains $2^{\aleph_{0}}$ non-$\mathscr{L}$\nbd related
  elements, as required.

  Now let $\Delta$~denote the graph that is the disjoint union of a
  copy~$R'$ of the countable universal graph and the empty graph~$E$
  (i.e., with no edges) on a countably infinite set of vertices.  By
  taking $\Gamma_{0} = \Delta$ in the initial step of our standard
  construction, we may assume that $\Delta = R' \disjunion E$ occurs
  as a subgraph of~$R$.  Let $g \colon R' \to R$ be a fixed
  isomorphism and let $h \colon E \to R$ be any map.  Then using
  Lemma~\ref{lem:map-extend}(i) we find an endomorphism $\xi_{h}
  \colon R \to R$ that simultaneously extends both $g$~and~$h$.  We
  continue to use the endomorphism~$f$ constructed above.  Note that
  $\xi_{h}f$~and~$f$ are $\mathscr{L}$\nbd related, for any choice
  of~$h$, since $g^{-1}\xi_{h}f = f$ (where by~$g^{-1}$ we mean the
  endomorphism of~$R$ corresponding to the inverse~$R \to R'$ of~$g$).

  Observe $\ker \xi_{h}f = \ker \xi_{h}$ since $f$~is injective.
  Therefore, if $h,k \colon E \to R$ are chosen with $\ker h \neq \ker
  k$, then $\xi_{h}f$~and~$\xi_{k}f$ are not $\mathscr{R}$\nbd
  related, by Lemma~\ref{lem:class-facts}(ii).  As there are
  $2^{\aleph_{0}}$~possible kernels for the map~$h$, we conclude the
  $\scD$\nbd class of~$f$ indeed contains $2^{\aleph_{0}}$~many
  $\mathscr{R}$\nbd classes.
\end{proof}

\begin{cor}
  \label{cor:graph-Duncount}
  There are $2^{\aleph_{0}}$~non-regular $\scD$\nbd classes in~$\End
  R$.
\end{cor}

\begin{proof}
  By Theorem~\ref{thm:ac-graphs} there are
  $2^{\aleph_{0}}$~isomorphism types of countable algebraically closed
  graphs.  By the previous theorem, for each such graph~$\Gamma$, with
  $\Gamma \not\cong R$, there is a non-regular injective
  endomorphism~$f$ with $\langle Vf \rangle \cong \Gamma$ and
  each~$\Gamma$ determines a distinct $\scD$\nbd class of~$f$ by
  Lemma~\ref{lem:class-facts}(iii).
\end{proof}

Finally, we turn to the \Schutz\ groups of $\scH$\nbd classes of
non-regular endomorphisms.  As mentioned in Section~\ref{sec:prelims},
for specific injective endomorphisms of the countable universal graph
we are able to make Proposition~\ref{prop:Schutzgps} more precise.

Let $\Gamma_{0} = (V_{0},E_{0})$ be a countable algebraically closed
graph.  Then, by Proposition~\ref{prop:characterise-ac}(i), there
exists some $F_{0} \subseteq E_{0}$ such that $(V_{0},F_{0})$~is
isomorphic to the countable universal graph~$R$.  Use~$\Gamma_{0}$ in
the initial step of the construction of~$R$.  Hence we can assume that
$R = (V,E)$ contains $\Gamma_{0} = (V_{0},E_{0})$ as a subgraph.  Let
$f \colon R \to R$ be the endomorphism that realises the isomorphism
$(V_{0},F_{0}) \cong R$; that is, $f$~is given by a bijection from~$V$
to~$V_{0}$ and from~$E$ to~$F_{0}$.

Consider a bijection $g \colon V_{0} \to V_{0}$ such that $g$~is an
automorphism both of $\im f = (V_{0},F_{0})$ and of $\langle Vf
\rangle = \langle V_{0} \rangle = (V_{0},E_{0})$.  By
Proposition~\ref{prop:Schutzgps}(iv), $fg$~is $\scL$\nbd related
to~$f$.  However, since $g$~is an automorphism of~$\Gamma_{0}$, we
can, by Lemma~\ref{lem:map-extend}(iii), extend it to an
automorphism~$\hat{g}$ of~$R$.  Then we observe that $fg$~and~$f$ are
also $\scR$\nbd related since $fg = f\hat{g}$ and $(fg)\hat{g}^{-1} =
f$.  We can now similarly establish that if $h$~is an element of the
$\scH$\nbd class~$H$ of~$f$, then $h\hat{g}$~is also $\scH$\nbd
related to~$h$.  Hence $\hat{g} \in T_{H}$ (in the notation introduced
in Section~\ref{sec:prelims}).  Now returning to
Proposition~\ref{prop:Schutzgps} we see that
\[
\gamma_{\hat{g}}\phi = \hat{g}|_{\langle V_{0} \rangle} = g.
\]
Hence we conclude that the image of~$\phi$ is $\Aut(V_{0},F_{0}) \cap
\Aut(V_{0},E_{0})$; that is:

\begin{prop}
  \label{prop:R-Schutz}
  Let $f$~be an injective endomorphism of the countable universal
  graph~$R$ of the form specified above and let $H = H_{f}$.  Then
  $\mathcal{S}_{H} \cong \Aut \langle Vf \rangle \cap \Aut(\im
  f)$. \qed
\end{prop}

To construct $\scH$\nbd classes with \Schutz\ group isomorphic to a
particular group, we need to specify the particular graph to select
as~$\Gamma_{0}$ in the above argument.  We shall again make use of the
graphs~$L_{S}$, for $S \subseteq \N \setminus \{0,1\}$, defined
earlier.  For such a subset~$S$, define~$M_{S}$ to be the graph whose
vertices are those of~$L_{S}$ together with new vertices~$x_{n}$
(for~$n \in \N$) and whose edges are those of~$L_{S}$ together with
additional edges
\[
\set{(y,x_{n}),(x_{n},y),(x_{m},x_{n}),(x_{n},x_{m})}{y \in V(L_{S}),
  \; m,n \in \N, \; m \neq n}.
\]
Note that the~$x_{n}$ are joined to every other vertex in~$M_{S}$,
while no other vertex has this property.  It follows that any
automorphism of~$M_{S}$ must induce an automorphism of~$L_{S}$ and
permute the vertices~$x_{n}$.  As $\Aut L_{S} = \1$, we conclude $\Aut
M_{S}$~is isomorphic to the symmetric group on a countably infinite
set.  Similarly, using Lemma~\ref{lem:L-props}(iii), if $S$~and~$T$
are subsets of~$\N \setminus \{0,1\}$ then $M_{S} \cong M_{T}$ if and
only if $S = T$.

Now let $\Gamma$~be an arbitrary countable graph and let $S_{n}$, for
$n \in \N$, be a sequence of distinct subsets of~$\N \setminus
\{0,1\}$ such that the graph~$M_{S_{n}}$ is not isomorphic to any connected
component of~$\Gamma$.  We perform the following construction: Define
$\Gamma_{0}^{\ast} = \Gamma^{\dagger}$ (the complement of~$\Gamma$, as
previously).  Then, assuming that $\Gamma_{n}^{\ast}$~has been
defined, enumerate the finite subsets of vertices
of~$\Gamma_{n}^{\ast}$ as~$(A_{i})_{i \in \N}$.  Let the vertices
of~$\Gamma_{n+1}^{\ast}$ be the union of the vertices
of~$\Gamma_{n}^{\ast}$, the vertices of~$L_{S_{n}}$ and new
vertices $\set{x^{(n)}_{i}}{i \in \N}$.  Define the edges
of~$\Gamma_{n+1}^{\ast}$ to be the edges of~$\Gamma_{n}^{\ast}$
together with edges between $a$~and~$x^{(n)}_{i}$ for all $a \in
A_{i}$ and all~$i$.  Having constructed the
graphs~$\Gamma_{n}^{\ast}$, we let $\Gamma^{\ast} =
(V^{\ast},E^{\ast})$~be the limit of this sequence of graphs.  By
construction, $\Gamma^{\ast}$~is existentially closed and therefore
isomorphic to the countable universal graph~$R$.

Now let $\Gamma_{0} = (V^{\ast},E_{0})$ be the graph whose edges are
all possible edges between pairs of vertices except the following are
\emph{not} included:
\begin{enumerate}
\item the edges in~$\Gamma$;
\item for each~$n \in \N$, all edges between distinct vertices of
  $\set{x^{(n)}_{i}}{i \in \N}$;
\item for each~$n \in \N$, the edges in~$L_{S_{n}}$;
\item for each~$n \in \N$, all edges between a vertex in~$L_{S_{n}}$
  and a vertex~$x^{(n)}_{i}$.
\end{enumerate}
Note then that $E^{\ast} \subseteq E_{0}$.  Therefore
$\Gamma_{0}$~is algebraically closed and it is this graph that we use
in the argument employed above to establish
Proposition~\ref{prop:R-Schutz}.  Let $f \colon R \to R$ be the
endomorphism given by an injective map on the set~$V$ of vertices
of~$R$ and whose image is~$(V^{\ast},E^{\ast})$.  Note that $f$~is
necessarily not regular by Proposition~\ref{prop:reg-image}, since
$\im f = (V^{\ast},E^{\ast}) \neq \langle Vf \rangle =
(V^{\ast},E_{0})$.  Then the \Schutz\ group of the $\scH$\nbd class
of~$f$ is as specified by Proposition~\ref{prop:R-Schutz}, namely
$\mathcal{S}_{H_{f}} \cong \Aut (V^{\ast},E_{0}) \cap \Aut
(V^{\ast},E^{\ast})$.

To apply this result, we first determine the automorphism group
of~$\Gamma_{0}$.  Note that $\Gamma_{0}^{\dagger}$~is the disjoint
union of the graphs~$\Gamma$ and~$M_{S_{n}}$ for $n \in \N$.  Hence
\[
\Aut \Gamma_{0} \cong \Aut \Gamma_{0}^{\dagger} \cong \Aut \Gamma
\times \prod_{n \in \N} \Aut M_{S_{n}} \cong \Aut \Gamma \times (\Sym
\N)^{\aleph_{0}}
\]
by our earlier observations.  It follows that if $g$~is a bijection $V
\to V$ that is simultaneously an automorphism of
both~$(V^{\ast},E_{0})$ and~$(V^{\ast},E^{\ast})$, then $g$~induces an
automorphism of~$\Gamma$, fixes all vertices of~$L_{S_{n}}$ (for all
$n \in \N$) and, for each~$n \in \N$, permutes the vertices in
$\set{x^{(n)}_{i}}{i \in \N}$.  However, given an automorphism
of~$\Gamma$, there is precisely one choice for these permutations that
defines an automorphism of~$(V^{\ast},E^{\ast})$ since, at each
stage~$n$, the vertex~$x^{(n)}_{i}$ must be mapped to the vertex
adjoined to the finite set~$A_{i}g$.  We conclude that mapping~$g$ to
its restriction to the vertices of~$\Gamma$ yields an isomorphism from
$\Aut(V^{\ast},E_{0}) \cap \Aut(V^{\ast},E^{\ast})$ to~$\Aut \Gamma$.

Finally, note that this method also constructs for us
$2^{\aleph_{0}}$~many $\scD$\nbd classes where the \Schutz\ group is
isomorphic to the automorphism group of~$\Gamma$.  First fix the
subsets~$S_{n}$ for~$n \geq 2$ as above.  There remain
$2^{\aleph_{0}}$~possible choices now for~$S_{1}$ in order to follow
the above construction.  Each such~$S_{1}$ determines an (injective)
non-regular endomorphism $f = f_{S_{1}}$ of~$R$ with
$\mathcal{S}_{H_{f}} \cong \Aut \Gamma$.  Moreover, since $\langle Vf
\rangle^{\dagger} = \Gamma_{0}^{\dagger}$ is the disjoint union
of~$\Gamma$ and the~$M_{S_{n}}$, when $S_{1} \neq S'_{1}$ there can
exist no isomorphism from~$\langle Vf_{S_{1}} \rangle$ to~$\langle
Vf_{S'_{1}} \rangle$ since $M_{S_{1}} \not\cong M_{S'_{1}}$.  Hence
$f_{S_{1}}$~and~$f_{S'_{1}}$ belong to distinct $\scD$\nbd classes if
$S_{1} \neq S'_{1}$ by Lemma~\ref{lem:class-facts}(iii).

We have now established our final result about the endomorphism monoid
of~$R$:

\begin{thm}
  \label{thm:R-uncountSchutz}
  Let $\Gamma$~be any countable graph.  There are $2^{\aleph_{0}}$
  non-regular $\scD$\nbd classes of the countable universal graph~$R$
  such that the \Schutz\ groups of $\scH$\nbd classes within them are
  isomorphic to~$\Aut \Gamma$.\qed
\end{thm}

\section{Directed graphs}
\label{sec:directed}

A \emph{directed graph} is a relational structure $\Gamma = (V,E)$
where $E$~is an irreflexive binary relation on~$V$.  This ensures that
every graph is, in particular, a directed graph.  We can therefore use
all the graphs constructed in Section~\ref{sec:graph} but now viewed
as directed graphs.  Consequently, our methods in this section are
almost identical to those for undirected graphs and we therefore omit
most of the details.  Furthermore, the class of groups arising as the
automorphism group of a graph are amongst those that arise as the
automorphism group of a directed graph.  Our first step is to show
that these two classes are actually the same.

Let $\Gamma = (V,E)$ be a countable directed graph.  Enumerate the
vertices of~$\Gamma$ as $V = \set{v_{i}}{i \in I}$ (where $I \subseteq
\N$).  We construct an (undirected) graph $\Gamma^{\dashv} = (V',E')$ as
follows.  Set
\[
V' = V \cup \set{ x_{jk}, y_{jk}, z_{jk} }{ (v_{j},v_{k}) \in E }
\]
and define the (undirected) edges in~$\Gamma^{\dashv}$ to be all
$(v_{j},x_{jk})$, $(x_{jk},y_{jk})$, $(y_{jk},z_{jk})$ and
$(y_{jk},v_{k})$ for all $(v_{j},v_{k}) \in E$.  This has the effect
of replacing the ``arrow'' from~$v_{j}$ to~$v_{k}$ in~$\Gamma$ by the
shape shown in Figure~\ref{fig:dashv}.

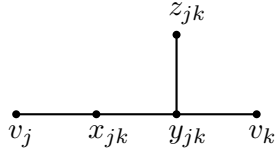
\begin{figure}
  \begin{center}
    \begin{pspicture}(0,.3)(3,1.2)
      \multiput(0,0)(1,0){4}{\qdisk(0,0){.05}}
      \qdisk(2,1){.05}
      \psline(0,0)(3,0)
      \psline(2,0)(2,1)
      \put(-.1,-.3){$v_{j}$}
      \put(.9,-.3){$x_{jk}$}
      \put(1.9,-.3){$y_{jk}$}
      \put(2.9,-.3){$v_{k}$}
      \put(1.9,1.25){$z_{jk}$}
    \end{pspicture}
  \end{center}
  \caption{Replacement edges in~$\Gamma^{\dashv}$}
  \label{fig:dashv}
\end{figure}

For our directed graph~$\Gamma = (V,E)$, let us also define, for $v
\in V$,
\[
\Gamma_{+}(v) = \set{x \in V}{(v,x) \in E}
\qquad \text{and} \qquad
\Gamma_{-}(v) = \set{x \in V}{(x,v) \in E}.
\]

\begin{lemma}
  \label{lem:dashv}
  Let $\Gamma = (V,E)$ be a directed graph.  Suppose that $\order{
    \Gamma_{+}(v) } + \order{ \Gamma_{-}(v) } > 3$ for all $v \in V$.
  Then $\Aut \Gamma^{\dashv} \cong \Aut \Gamma$.
\end{lemma}

\begin{proof}
  We observe that, by construction, in~$\Gamma^{\dashv}$, every vertex
  in $X = \set{x_{jk}}{(v_{j},v_{k}) \in E}$ has degree~$2$, every
  vertex in $Y = \set{y_{jk}}{(v_{j},v_{k}) \in E}$ has degree~$3$,
  every vertex in $Z = \set{z_{jk}}{(v_{j},v_{k}) \in E}$ has
  degree~$1$ and (by assumption) every~$v_{i}$ has degree greater
  than~$3$.  Consequently, if $f \in \Aut \Gamma^{\dashv}$ then $Xf =
  X$, $Yf = Y$, $Zf = Z$ and $Vf = V$.

  Let us define a map $\phi \colon \Aut \Gamma^{\dashv} \to \Aut
  \Gamma$ by $f\phi = f|_{V}$ for all $f \in \Aut \Gamma^{\dashv}$.
  For such an automorphism~$f$, from the above observation,
  $f\phi$~defines a bijection $V \to V$.  If $(u,v)$~is a (directed)
  edge in~$\Gamma$, then there exists $x \in X$ and $y \in Y$ such
  that $(u,x)$,~$(x,y)$ and~$(y,v)$ are (undirected) edges
  in~$\Gamma^{\dashv}$.  Then $(uf,xf)$,~$(xf,yf)$ and~$(yf,vf)$ are
  edges in~$\Gamma^{\dashv}$ and necessarily $xf \in X$ and $yf \in
  Y$.  It follows that $(uf,vf)$~must be an edge in~$\Gamma$.
  Similarly, if $(u,v) \notin E$, then $(uf,vf) \notin E$.  Hence
  $f\phi$~is indeed a graph automorphism of~$\Gamma$.

  It is straightforward to see that $\phi$~is a homomorphism and,
  since the images of $x_{jk}$,~$y_{jk}$ and~$z_{jk}$ (for
  $(v_{j},v_{k}) \in E$) under~$f$ are completely determined by
  $v_{j}f$~and~$v_{k}f$, it follows that $\phi$~is injective.
  Finally, if $h \in \Aut \Gamma$, define $v_{i}f = v_{i}h$ for
  all~$i$ and if $(v_{j},v_{k}) \in E$ with $v_{j}h = v_{\ell}$ and
  $v_{k}h = v_{m}$ define $x_{jk}f = x_{\ell m}$, \ $y_{jk}f = y_{\ell
    m}$ and $z_{jk}f = z_{\ell m}$.  This defines $f \in \Aut
  \Gamma^{\dashv}$ with the property that $f\phi = f|_{V} = h$.
  Consequently, $\phi$~is an isomorphism as required.
\end{proof}

In addition to the above lemma, the other tools we require are the
constructions used in Section~\ref{sec:graph}.  If $S$~is a subset
of~$\N \setminus \{0,1\}$, we define the graph~$L_{S}$ as earlier, but
we now view it as a directed graph.  So, for any pair of vertices
$u$~and~$v$ joined in~$L_{S}$, there is both an edge from~$u$ to~$v$
and from~$v$ to~$u$ (see Figure~\ref{fig:L-dirgraph}).  The disjoint
union of two directed graphs and the complement~$\Gamma^{\dagger}$ of
a directed graph~$\Gamma$ are defined exactly as earlier.  In
particular, there is an edge~$(u,v)$ from~$u$ to~$v$
in~$\Gamma^{\dagger}$ if and only if there is no edge from~$u$ to~$v$
in~$\Gamma$.  Using these we now observe that the classes of groups
arising as the automorphism group of a directed graph and as the
automorphism group of an undirected graph coincide.

\begin{figure}
  \begin{center}
    \begin{pspicture}(0,.5)(7,1)
      \psset{arrowscale=1.8, linewidth=.6pt}
      \multips(0,0)(1,0){7}{\qdisk(0,0){.05}}
      \qdisk(2,1){.05}
      \qdisk(4,1){.05}
      \qdisk(5,1){.05}
      \put(6.3,0){\dots}
      \pscurve{->}(0,0)(.5,.2)(1,0)
      \pscurve{->}(1,0)(.5,-.2)(0,0)
      \pscurve{->}(1,0)(1.5,.2)(2,0)
      \pscurve{->}(2,0)(1.5,-.2)(1,0)
      \pscurve{->}(2,0)(2.5,.2)(3,0)
      \pscurve{->}(3,0)(2.5,-.2)(2,0)
      \pscurve{->}(3,0)(3.5,.2)(4,0)
      \pscurve{->}(4,0)(3.5,-.2)(3,0)
      \pscurve{->}(4,0)(4.5,.2)(5,0)
      \pscurve{->}(5,0)(4.5,-.2)(4,0)
      \pscurve{->}(5,0)(5.5,.2)(6,0)
      \pscurve{->}(6,0)(5.5,-.2)(5,0)
      \pscurve{->}(2,0)(1.8,.5)(2,1)
      \pscurve{->}(2,1)(2.2,.5)(2,0)
      \pscurve{->}(4,0)(3.8,.5)(4,1)
      \pscurve{->}(4,1)(4.2,.5)(4,0)
      \pscurve{->}(5,0)(4.8,.5)(5,1)
      \pscurve{->}(5,1)(5.2,.5)(5,0)
      \pscurve{-}(6,0)(6.15,.09)(6.3,.15)
      \pscurve{<-}(6,0)(6.15,-.09)(6.3,-.15)
    \end{pspicture}
  \end{center}
  \caption{$L_{\{2,4,5,\dots\}}$~as a directed graph}
  \label{fig:L-dirgraph}
\end{figure}
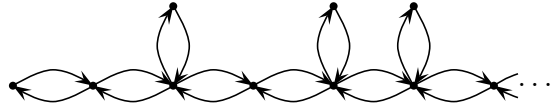

\begin{prop}
  \label{prop:directed-automs}
  Let\/ $\Gamma$~be a countable directed graph.  Then there exists an
  (undirected) countable graph~$\Lambda$ such that $\Aut \Lambda \cong
  \Aut \Gamma$.
\end{prop}

\begin{proof}
  As $\Gamma$~is countable, we can choose~$S$ such that the directed
  graph~$L_{S}$ is not isomorphic to any component of~$\Gamma$.  Take
  $\Delta_{S}$~to be the directed graph $(\Gamma \disjunion
  L_{S})^{\dagger}$ and $\Lambda$~to be the undirected
  graph~$\Delta_{S}^{\dashv}$ as constructed above.  By construction,
  each vertex in~$\Delta_{S}$ has infinite degree and so, by
  Lemmas~\ref{lem:dashv}, \ref{lem:complement}(i)
  and~\ref{lem:L-props}(i), $\Aut \Lambda \cong \Aut \Delta_{S} \cong
  \Aut (\Gamma \disjunion L_{S}) \cong \Aut \Gamma$, as required.
\end{proof}

For the class of directed graphs, the condition to be algebraically
closed is easily seen to be equivalent to the following: a directed
graph $\Gamma = (V,E)$ is \emph{algebraically closed} if for any
finite subset~$A$ of its vertices, there exists some vertex~$v$ such
that $(u,v),(v,u) \in E$ for all $u \in A$.  If we start with a
directed graph~$\Gamma$ and perform the same constructions as in
Section~\ref{sec:graph}, then we observe that $\Delta_{S} = (\Gamma
\disjunion L_{S})^{\dagger}$ is algebraically closed (as a directed
graph) and, provided $L_{S}$~is not isomorphic to any (weakly)
connected component of~$\Gamma$, that $\Aut \Delta_{S} \cong \Aut
\Gamma$.  This establishes the analogue of Theorem~\ref{thm:ac-graphs}
for the class of directed graphs:

\begin{thm}
  \label{thm:ac-directed}
  Let\/ $\Gamma$~be a countable (undirected) graph.  Then there exist
  $2^{\aleph_{0}}$~pairwise non-isomorphic countable algebraically closed
  directed graphs whose automorphism group is isomorphic to that
  of\/~$\Gamma$. \qed
\end{thm}

As in the previous section, we shall make use of a standard method to
construct a copy of the countable universal directed graph.  If
$\Gamma = (V,E)$~is any countable directed graph, first construct a
new directed graph~$\mathcal{H}(\Gamma)$ as follows.  Enumerate the
set of all triples of finite and pairwise disjoint subsets of~$\Gamma$
as $(A_{i},B_{i},C_{i})_{i \in I}$ for $I \subseteq \N$.  Define
$\mathcal{H}(\Gamma)$~to be the directed graph with vertex set $V \cup
\set{v_{i}}{i \in I}$ (where each~$v_{i}$ is a new vertex) and edge
set
\[
E \cup \set{(v_{i},a),(b,v_{i}),(v_{i},c),(c,v_{i})}{a \in A_{i}, b
  \in B_{i}, c \in C_{i}, i \in I}.
\]
Thus each new vertex~$v_{i}$~has the property that there is an edge
from~$v_{i}$ to every vertex in~$A_{i}$, from every vertex in~$B_{i}$
to~$v_{i}$, and from~$v_{i}$ to every vertex in~$C_{i}$ and \emph{vice
  versa}.

Now construct a sequence of directed graphs~$\Gamma_{n}$ by defining
$\Gamma_{0} = \Gamma$ and $\Gamma_{n+1} = \mathcal{H}(\Gamma_{n})$ for
each~$n \geq 0$.  We define $\Gamma_{\infty}$~to be the limit of this
sequence of graphs, which is countable and, by construction,
existentially closed.  Therefore $\Gamma_{\infty}$~is isomorphic to
the countable universal directed graph~$D$.

The same arguments apply to constructing extensions of homomorphisms
and an analogue of Lemma~\ref{lem:map-extend} transfers straight
across to the setting of directed graphs.  We then establish, by
identical methods as used in Section~\ref{sec:graph}, first the
characterization of graphs that arise as images of idempotent
endomorphisms of the countable universal directed graph~$D$ and second
observations about $\scH$- and $\scD$\nbd classes of regular elements
of the endomorphism monoid~$D$:

\begin{thm}
  \label{thm:directed-idemp}
  Let $\Gamma$~be a countable directed graph.  Then
  \begin{enumerate}
  \item there exists an idempotent endomorphism~$f$ of the countable
    universal directed graph~$D$ such that $\im f \cong \Gamma$ if and
    only if\/ $\Gamma$~is algebraically closed;
  \item if\/ $\Gamma$~is algebraically closed, there are
    $2^{\aleph_{0}}$~idempotent endomorphisms~$f$ of~$D$ such that
    $\im f \cong \Gamma$. \qed
  \end{enumerate}
\end{thm}

\begin{thm}
  \label{thm:D-reg-classes}
  Let $D$~denote the countable universal directed graph.
  \begin{enumerate}
  \item Let\/ $\Gamma$~be a countable (undirected) graph.  Then there
    exist $2^{\aleph_{0}}$ distinct regular $\scD$\nbd classes
    of\/~$\End D$ whose group $\scH$\nbd classes are isomorphic
    to~$\Aut \Gamma$.
  \item Every regular $\scD$\nbd class of\/~$\End D$ contains
    $2^{\aleph_{0}}$~distinct group $\scH$\nbd classes. \qed
  \end{enumerate}
\end{thm}

It is a consequence of this theorem that the maximal subgroups of the
endomorphism monoid of the countable universal directed graph are, up
to isomorphism, the same as the maximal subgroups of the endomorphism
monoid of the countable universal (undirected) graph~$R$.

\spc

The formulae used in Section~\ref{sec:graph} to
define~$\Gamma^{\sharp}$ make perfect sense when $\Gamma$~is a
directed graph.  They provide us with a directed graph satisfying
analogous properties to those stated in Lemma~\ref{lem:sharp}.  This
is the primary tool in establishing the following results, which are
the directed graph analogues of~\ref{thm:graph-regLR},
\ref{thm:graph-nonregLR} and~\ref{cor:graph-Duncount}.  The proofs are
identical to those for undirected graphs, except for judicious
insertion of the word ``directed'', and so we omit them.

\begin{thm}
  \label{thm:directed-regLR}
  Every regular $\scD$\nbd class of the endomorphism monoid of the
  countable universal directed graph~$D$ contains
  $2^{\aleph_{0}}$~many $\scL$- and $\scR$\nbd classes.
\end{thm}

\begin{thm}
  \label{thm:directed-nonregLR}
  Let $\Gamma$~be any algebraically closed directed graph that is not
  isomorphic to the countable universal directed graph~$D$.  Then
  there exists a non-regular injective endomorphism~$f$ of~$D$ such
  that the subgraph induced by the images of the vertices under~$f$ is
  isomorphic to~$\Gamma$ and such that the $\scD$\nbd class of~$f$
  contains $2^{\aleph_{0}}$ $\scR$- and $\scL$\nbd classes. \qed
\end{thm}

\begin{cor}
  There are $2^{\aleph_{0}}$ non-regular $\scD$\nbd classes in~$\End
  D$. \qed
\end{cor}

For the \Schutz\ groups of $\scH$\nbd classes of non-regular
endomorphisms of the countable universal directed graph $D = (V,E)$,
we proceed once more as in Section~\ref{sec:graph}.  If $\Gamma_{0} =
(V_{0},E_{0})$ is an algebraically closed directed graph, let $F_{0}
\subseteq E_{0}$ be such that $(V_{0},F_{0}) \cong D$ (as provided by
Proposition~\ref{prop:characterise-ac}(i)).  Assume that $D$~has been
constructed using $\Gamma_{0}$~in the initial step of our standard
method and let $f \colon D \to D$ be the endomorphism that realises
this isomorphism; that is, $f$~is given by a bijection $V \to V_{0}$
that induces a bijection from~$E$ to~$F_{0}$.  Then the same argument
as used to establish Proposition~\ref{prop:R-Schutz} gives:

\begin{prop}
  \label{prop:D-Schutz}
  Let $f$~be an injective endomorphism of the countable universal
  directed graph~$D$ of the form specified above and $H = H_{f}$.
  Then $\mathcal{S}_{H} \cong \Aut \langle Vf \rangle \cap \Aut( \im
  f)$. \qed
\end{prop}

To apply this proposition, we make minor changes in the argument used to
establish Theorem~\ref{thm:R-uncountSchutz}.  Let $\Gamma$~be an
arbitrary countable (undirected) graph and let $S_{n}$, for $n \in
\N$, be a sequence of distinct subsets of~$\N \setminus \{0,1\}$ such
that the graph~$M_{S_{n}}$ (as defined towards the end of
Section~\ref{sec:graph}) is not isomorphic to any connected component
of~$\Gamma$.  Define $\Gamma_{0}^{\ast} = \Gamma^{\dagger}$ (the
complement of~$\Gamma$) and view this as a directed graph.  Then,
assuming that the directed graph~$\Gamma_{n}^{\ast}$ has been defined,
enumerate the triples of finite pairwise disjoint subsets
of~$\Gamma_{n}^{\ast}$ as $(A_{i},B_{i},C_{i})_{i \in \N}$.  Let the
vertices of~$\Gamma_{n+1}^{\ast}$ be the union of the vertices
of~$\Gamma_{n}^{\ast}$, the vertices of the graph~$L_{S_{n}}$ and new
vertices $\set{x_{i}^{(n)}, y_{i}^{(n)}, z_{i}^{(n)}}{i \in \N}$.
Define the edges of~$\Gamma_{n+1}^{\ast}$ to be the edges
of~$\Gamma_{n}^{\ast}$ together with, for all $i \in \N$,
\ $(x_{i}^{(n)},a)$ for $a \in A_{i}$, \ $(b,y_{i}^{(n)})$ for $b \in
B_{i}$, and both $(z_{i}^{(n)},c)$ and~$(c,z_{i}^{(n)})$ for $c \in
C_{i}$.  We define $\Gamma^{\ast} = (V^{\ast},E^{\ast})$ to be the
limit of this sequence of directed graphs, which is existentially
closed by construction and so isomorphic to the countable universal
directed graph~$D$.

Then let $\Gamma_{0} = (V^{\ast},E_{0})$ be the directed graph whose
edges are all possible edges between pairs of vertices except the
following are \emph{not} included:
\begin{enumerate}
\item the edges of~$\Gamma$;
\item for each $n \in \N$, all edges between distinct vertices of
  $\set{x_{i}^{(n)}, y_{i}^{(n)}, z_{i}^{(n)}}{i \in \N}$;
\item for each $n \in \N$, the edges in~$L_{S_{n}}$;
\item for each $n \in \N$, all edges between a vertex in~$L_{S_{n}}$
  and a vertex~$x_{i}^{(n)}$, $y_{i}^{(n)}$ or~$z_{i}^{(n)}$ and
  \emph{vice versa}.
\end{enumerate}
We have, of course, constructed an undirected graph but in this
context we shall view $\Gamma_{0}$~as a directed graph.  By
construction, $E^{\ast} \subseteq E_{0}$.  Therefore $\Gamma_{0}$~is
an algebraically closed directed graph and we use this graph when
applying Proposition~\ref{prop:D-Schutz}.  The endomorphism $f \colon
D \to D$ whose image is~$(V^{\ast},E^{\ast})$ is not regular since
$E^{\ast} \neq E_{0}$.

The complement~$\Gamma_{0}^{\dagger}$ is the disjoint union
of~$\Gamma$ and copies of~$M_{S_{n}}$ (namely the graph comprising the
vertices of~$L_{S_{n}}$ and the vertices $x_{i}^{(n)}$,~$y_{i}^{(n)}$
and~$z_{i}^{(n)}$ for $i \in \N$).  Hence
\[
\Aut \Gamma_{0} \cong \Aut \Gamma_{0}^{\dagger} \cong \Aut \Gamma
\times \prod_{n \in \N} \Aut M_{S_{n}} \cong \Aut \Gamma \times (\Sym
\N)^{\aleph_{0}}.
\]
The same argument as employed in Section~\ref{sec:graph} shows that
$\Aut(V^{\ast},E_{0}) \cap \Aut(V^{\ast},E^{\ast}) \cong \Aut
\Gamma$.  Equally, by varying the subset~$S_{1}$ we produce
$2^{\aleph_{0}}$~many $\scD$\nbd classes of the endomorphism~$f$.
Therefore, the analogue of Theorem~\ref{thm:R-uncountSchutz} holds for
the countable universal directed graph:

\begin{thm}
  \label{thm:D-uncountSchutz}
  Let\/ $\Gamma$~be any countable (undirected) graph.  There are
  $2^{\aleph_{0}}$ non-regular $\scD$-classes of the countable
  universal directed graph~$D$ such that the \Schutz\ groups of
  $\scH$\nbd classes within them are isomorphic to~$\Aut \Gamma$.\qed
\end{thm}

\section{Bipartite graphs}
\label{sec:bipartite}

One usually defines an (undirected) graph $\Gamma = (V,E)$ to be
bipartite if there exists a function $c \colon V \to \{0,1\}$ such
that $c(u) \neq c(v)$ whenever $(u,v) \in E$.  However, as noted by
Evans in~\cite[Section~2.2.2]{Evans}, it is easy to observe that the
class of finite graphs satisfying this condition does not have the
amalgamation property and so we cannot speak of the Fra\"{\i}ss\'{e}
limit of such graphs.  The solution is to encode the partition of the
vertex set~$V$ via an additional equivalence relation.  Accordingly,
we define a \emph{bipartite graph} to be a relational structure
$\Gamma = (V,E,P)$ such that $E$~is an irreflexive symmetric binary
relation on~$V$, \ $P = (V_{0} \times V_{0}) \cup (V_{1} \times
V_{1})$ for some partition $V = V_{0} \disjunion V_{1}$ of the vertex
set, and if $(u,v) \in E$ then $(u,v) \not\in P$ (which means one of
$u$~and~$v$ belongs to~$V_{0}$ and the other belongs to~$V_{1}$).

Let $\Gamma = (V,E,P)$ and $\Delta = (W,F,Q)$ be bipartite graphs in
this sense, where $V = V_{0} \disjunion V_{1}$ and $W = W_{0}
\disjunion W_{1}$ are the partitions of the vertex sets given by
$P$~and~$Q$, respectively.  It follows from the definition that if $f
\colon \Gamma \to \Delta$ is a homomorphism, then $V_{0}f$~is
contained in one of $W_{0}$~or~$W_{1}$ and similarly for~$V_{1}f$.
Moreover, if $f$~is an embedding then either $V_{0}f \subseteq W_{0}$
and $V_{1}f \subseteq W_{1}$, or $V_{0}f \subseteq W_{1}$ and $V_{1}f
\subseteq W_{0}$.  This enables one to show that the class of
bipartite graphs (according to our definition) satisfies the
amalgamation property and hence there is a unique Fra\"{\i}ss\'{e}
limit of the finite bipartite graphs, the \emph{countable universal
  bipartite graph}~$B$.

A particular observation from the previous paragraph is that if $f$~is
an automorphism of the bipartite graph $\Gamma = (V,E,P)$, where $P =
(V_{0} \times V_{0}) \cup (V_{1} \times V_{1})$, then either $V_{0}f =
V_{0}$ and $V_{1}f = V_{1}$, or $V_{0}f = V_{1}$ and $V_{1}f = V_{0}$.
We call~$f$ \emph{part-fixing} if $V_{0}f = V_{0}$ and $V_{1}f =
V_{1}$.  The following observation is straightforward.

\begin{lemma}
  \label{lem:autom-w/o-bp}
  Let $\Gamma = (V,E,P)$ be a bipartite graph such that the
  graph~$(V,E)$ is connected.  Then $\Aut \Gamma = \Aut (V,E)$.
\end{lemma}

To address which groups could arise as the group $\scH$\nbd classes
and \Schutz\ groups of the countable universal bipartite graph, we
shall first observe that any group arising as the automorphism group
of a graph also arises as that of a bipartite graph, and \emph{vice
  versa}.  To achieve the first of these, let $\Gamma = (V,E)$ be any
countable graph.  We now associate a countable bipartite graph
$\Gamma' = (V',E',P')$ to~$\Gamma$.  Enumerate the vertices
of~$\Gamma$ as $V = \set{v_{i}}{i \in I}$ (where $I \subseteq \N$).
Whenever there is an edge joining vertices $v_{i}$~and~$v_{j}$
in~$\Gamma$ (with $i < j$), choose a new vertex~$x_{ij}$ so that the
set $X = \set{x_{ij}}{(v_{i},v_{j}) \in E, \: i < j}$ is a set
disjoint from~$V$.  Set $V' = V \cup X$, \ $P' = (V \times V) \cup (X
\times X)$ and
\[
E' = \set{(v_{i},x_{ij}), (v_{j},x_{ij}), (x_{ij},v_{i}),
  (x_{ij},v_{j})}{ (v_{i},v_{j}) \in E, \: i < j}.
\]
(Intuitively, we have added a new vertex~$x_{ij}$ in the middle of
each original edge~$(v_{i},v_{j})$ and the partition of our new
graph~$\Gamma'$ is into ``old vertices'' and ``new middle-edge
vertices''.)

Define $\Gamma(v)$~to be the set of vertices in~$\Gamma$ that are
joined by an edge to~$v$.

\begin{lemma}
  \label{lem:Aut-Gamma'}
  Let\/ $\Gamma = (V,E)$~be a countable graph such that\/
  $\order{\Gamma(v)} \geq 3$ for all $v \in V$.  Then $\Aut \Gamma'
  \cong \Aut \Gamma$ and every automorphism of\/~$\Gamma'$ is
  part-fixing.
\end{lemma}

\begin{proof}
  If $f \in \Aut \Gamma'$, then $f$~defines an automorphism of the
  graph~$(V',E')$.  Since each $x \in X$ has degree~$2$ and each
  vertex $v \in V$ has the same degree in~$\Gamma'$ as in~$\Gamma$, we
  conclude from the hypothesis that $Xf = X$ and $Vf = V$.  In other
  words, each automorphism of~$\Gamma'$ is part-fixing.

  Now define a map $\phi \colon \Aut \Gamma' \to \Aut \Gamma$ by
  defining~$f\phi$ to be the restriction of~$f$ to the elements
  of~$V$.  We have observed that $f$~restricts to a bijection $V \to
  V$.  If $(u,v) \in E$, then there is a unique~$x \in X$ such that
  $(u,x),(x,v) \in E'$.  Then $(uf,xf),(xf,vf) \in E'$, as $f$~is an
  automorphism of~$\Gamma'$, and the definition of our bipartite
  graph~$\Gamma'$ then tells us that $(uf,vf) \in E$.  Similarly, if
  $(uf,vf) \in E$, then we deduce $(u,v) \in E$ and we conclude that
  $f\phi$~does indeed define an automorphism of the graph~$\Gamma$.

  Then $\phi$~is a homomorphism from~$\Aut \Gamma'$ to~$\Aut \Gamma$.
  It is injective, since the effect of~$f$ on the vertices in~$V$
  completely determines the effect on the vertices in~$X$.  If $h \in
  \Aut \Gamma$, we may extend~$h$ to an automorphism~$\tilde{h}$
  of~$\Gamma'$ by mapping a vertex~$x \in X$ satisfying $(u,x),(x,v)
  \in E$ to the unique vertex~$y \in X$ satisfying $(uh,y),(y,vh) \in
  E$.  Then $\tilde{h} \in \Aut \Gamma'$ and $\tilde{h}\phi = h$,
  completing the proof.
\end{proof}

We shall also need analogues of the constructions in
Section~\ref{sec:graph} for bipartite graphs.  First, if $\Gamma =
(V,E,P)$ is a bipartite graph, with $P = (V_{0} \times V_{0}) \cup
(V_{1} \times V_{1})$, define the \emph{bipartite
  complement}~$\Gamma^{\ddagger}$ to be $\Gamma^{\ddagger} =
(V,E^{\ddagger},P)$ where
\[
E^{\ddagger} = (V \times V) \setminus ( E \cup P ).
\]
Then by construction, $\Gamma^{\ddagger}$~is a bipartite graph with
the same vertex partition as~$\Gamma$ such that, for $u \in V_{0}$ and
$v \in V_{1}$, \ $(u,v)$~is an edge in~$\Gamma^{\ddagger}$ if and only
if $(u,v)$~is not an edge in~$\Gamma$.

If $\Gamma = (V,E,P)$, where $P = (V_{0} \times V_{0}) \cup (V_{1}
\times V_{1})$, and $\Delta = (V',E',P')$, where $P' = (V'_{0} \times
V'_{0}) \cup (V'_{1} \times V'_{1})$, are bipartite graphs, then we
define the \emph{bipartite disjoint union} of $\Gamma$~and~$\Gamma'$
to be
\[
\Gamma \disjsqunion \Delta = (V \cup V', E \cup E', Q)
\]
where $Q = (V_{0} \cup V'_{0}) \times (V_{0} \cup V'_{0}) \cup (V_{1}
\cup V'_{1}) \times (V_{1} \cup V'_{1})$.  This definition depends
upon our choice for $V_{0}$,~$V_{1}$, $V'_{0}$ and~$V'_{1}$ (i.e.,
which way round we pair the parts of the vertex sets) so strictly
speaking we must refer to \emph{a} bipartite disjoint union of
$\Gamma$~and~$\Delta$.  This choice, however, does not affect the
results in this section.

Finally, we shall use the following adjustment to produce a bipartite
analogue of the graph~$L_{S}$ defined in Section~\ref{sec:graph}.  For
$S \subseteq \N \setminus \{ 0,1 \}$, take $\Lambda_{S} =
(V_{S},E_{S},P_{S})$, where $V_{S}$~and~$E_{S}$ are the vertices and
edges of~$L_{S}$ as defined in Section~\ref{sec:graph} and $P_{S} =
(V_{0} \times V_{0}) \cup (V_{1} \times V_{1})$ where
\begin{align*}
V_{0} &= \set{\ell_{n}}{\text{$n$~is even}} \cup \set{v_{n}}{\text{$n
    \in S$ is odd}} \\
V_{1} &= \set{\ell_{n}}{\text{$n$~is odd}} \cup \set{v_{n}}{\text{$n
      \in S$ is even}}.
\end{align*}

With these constructions, we easily establish analogues of
Lemmas~\ref{lem:complement} and~\ref{lem:L-props}:  If
$\Gamma$~and~$\Delta$ are bipartite graphs then $\Aut
\Gamma^{\ddagger} = \Aut \Gamma$ and $\Gamma \cong \Delta$ if and only
if $\Gamma^{\ddagger} \cong \Delta^{\ddagger}$.  Also if $S$~and~$T$
are subsets of~$\N \setminus \{0,1\}$, then $\Aut \Lambda_{S} = \1$
and $\Lambda_{S} \cong \Lambda_{T}$ if and only if $S = T$.

\begin{thm}
  \label{thm:bipartite-automs}
  The class of groups arising as automorphism groups of countable
  graphs is precisely the same as the class of groups arising as
  automorphism groups of countable bipartite graphs.
\end{thm}

\begin{proof}
  Let $\Gamma$~be a countable graph.  Choose $S \subseteq \N \setminus
  \{0,1\}$ such that the graph~$L_{S}$, as in Section~\ref{sec:graph}
  is not isomorphic to any connected component of~$\Gamma$.  Let
  $\Delta_{S} = (\Gamma \disjunion L_{S})^{\dagger}$, so that $\Aut
  \Delta_{S} \cong \Aut \Gamma$ (see Theorem~\ref{thm:ac-graphs}).
  Then each vertex~$v$ of~$\Delta_{S}$ has infinite degree and so,
  taking $\Lambda = (\Delta_{S})'$ as defined above, we conclude from
  Lemma~\ref{lem:Aut-Gamma'} that $\Aut \Lambda \cong \Aut \Delta_{S}
  \cong \Aut \Gamma$.

  Conversely, if $\Gamma = (V,E,P)$~is a bipartite graph, choose $S
  \subseteq \N \setminus \{0,1\}$ such that $L_{S}$~is not isomorphic
  to any connected component of the graph~$(V,E)$.  Take
  $\Delta_{S} = (W,F,Q)$~to be the bipartite graph $(\Gamma
  \disjsqunion \Lambda_{S})^{\ddagger}$.  Then $(W,F)$~is a connected
  graph and so, by Lemma~\ref{lem:autom-w/o-bp}, $\Aut (W,F) = \Aut
  \Delta_{S} \cong \Aut \Gamma$.
\end{proof}

The definition of algebraic closure for bipartite graphs is soon
established to be equivalent to the following.  A bipartite graph
$\Gamma = (V,E,P)$ is \emph{algebraically closed} if, given a finite
collection of vertices $\{ v_{1}$,~$v_{2}$, \dots,~$v_{k} \}$ such
that $(v_{i},v_{j}) \in P$ for all $i$~and~$j$, there exists some
vertex~$w$ that is connected by an edge to each of the~$v_{i}$.
Equivalently, if $V = V_{0} \disjunion V_{1}$ is the partition of the
vertices determined by the relation~$P$, then if given finite subsets
$A_{0} \subseteq V_{0}$ and $A_{1} \subseteq V_{1}$, there exist
vertices $w_{0} \in V_{0}$ and $w_{1} \in V_{1}$ such that $w_{0}$~is
joined by an edge for each vertex in~$A_{1}$ and $w_{1}$~is joined by
an edge to each vertex in~$A_{0}$.

There are some potentially unexpected consequences of this
observation.  The first is that, for $m,n \in \N$, the complete
bipartite graph~$K_{m,n}$ on two parts of cardinalities $m$~and~$n$,
respectively, is algebraically closed.  There are also examples of
algebraically closed infinite bipartite graphs but with only one
witness~$w$ of the algebraic closure condition.  This stands in
contrast to the observation that in both algebraically closed graphs
and algebraically closed directed graphs, there are always infinitely
many witnesses.  Accordingly, we define a bipartite graph~$\Gamma =
(V,E,P)$ to be \emph{strongly algebraically closed} if for every
finite collection~$A$ of vertices satisfying $(v,w) \in P$ for all
distinct $v,w \in V$, there exist infinitely many vertices~$x$
connected to every vertex of~$A$ by an edge.  Our analogue of
Theorem~\ref{thm:ac-graphs} for bipartite graphs makes use of this
stronger condition.

\begin{thm}
  \label{thm:ac-bipartite}
  Let\/ $\Gamma$~be a countable graph.  Then there exist
  $2^{\aleph_{0}}$ pairwise non-isomorphic strongly algebraically
  closed bipartite graphs whose automorphism group is isomorphic to
  that of\/~$\Gamma$.
\end{thm}

\begin{proof}
  By Theorem~\ref{thm:bipartite-automs}, there is a countable
  bipartite graph~$\Delta$ such that $\Aut \Delta \cong \Aut \Gamma$.
  Now there are $2^{\aleph_{0}}$~choices for subsets~$S$ of~$\N
  \setminus \{0,1\}$ such that $\Lambda_{S}$~is not isomorphic to any
  connected component of~$\Delta$.  Take $\Pi_{S} = (\Delta
  \disjsqunion \Lambda_{S})^{\ddagger}$.  Then, by construction, $\Aut
  \Pi_{S} \cong \Aut \Delta \times \Aut \Lambda_{S} \cong \Aut
  \Gamma$, while $\Pi_{S} \cong \Pi_{T}$ if and only if $S = T$.
  Moreover, each~$\Pi_{S}$ is strongly algebraically closed: given any
  finite subset~$A$ of vertices of~$\Pi_{S}$ that are related under
  the partition relation, there are in fact infinitely many vertices
  in~$\Lambda_{S}$ that are joined to every vertex in~$A$.
\end{proof}

The standard method to construct a copy of the countable universal
bipartite graph is as follows.  If $\Gamma = (V,E,P)$ is any countable
bipartite graph with corresponding vertex partition $V = V_{0}
\disjunion V_{1}$, enumerate the set of all finite subsets of~$V_{0}$
as~$(A_{i})_{i \in I}$ and all finite subsets of~$V_{1}$
as~$(B_{j})_{j \in J}$ for some $I,J \subseteq \N$.  Set $W_{0} =
V_{0} \cup \set{v_{j}}{j \in J}$ and $W_{1} = V_{1} \cup \set{w_{i}}{i
  \in I}$, where the $v_{j}$~and~$w_{i}$ are new vertices.  Define
$\mathcal{I}(\Gamma)$ to be the bipartite graph with vertex set~$W_{0}
\cup W_{1}$, edge set to consist of~$E$ and new edges joining
each~$v_{j}$ to every element of~$B_{j}$ (for $j \in J$) and
each~$w_{i}$ to every element of~$A_{i}$ (for $i \in I$), and
partition relation
\[
Q = (W_{0} \times W_{0}) \cup (W_{1} \times W_{1}).
\]
As in the previous sections, we construct a sequence of bipartite
graphs by setting $\Gamma_{0} = \Gamma$ and $\Gamma_{n+1} =
\mathcal{I}(\Gamma_{n})$.  Then the limit~$\Gamma_{\infty}$ of this
sequence is existentially closed and therefore isomorphic to the
countable universal bipartite graph~$B$.

Many of our arguments transfer from Section~\ref{sec:graph} but
throughout one needs to be careful when, for example, finite
algebraically closed bipartite graphs arise.  The first example of
this occurs in our analogue of Lemma~\ref{lem:map-extend} for
bipartite graphs.

\begin{lemma}
  \label{lem:bipartite-extend}
  Let $\Gamma$~be a countable bipartite graph, let
  $\Gamma_{\infty}$~be the copy of the countable universal bipartite
  graph constructed around~$\Gamma$ as described above, and let $f
  \colon \Gamma \to \Gamma_{\infty}$ be a homomorphism of bipartite
  graphs.
  \begin{enumerate}
  \item There exist $2^{\aleph_{0}}$~endomorphisms $\hat{f} \colon
    \Gamma_{\infty} \to \Gamma_{\infty}$ such that the restriction
    of~$\hat{f}$ to~$\Gamma$ equals~$f$.
  \item If\/ $\im f$~is an algebraically closed bipartite graph, then
    there is an extension~$\hat{f}$ of~$f$ with $\im \hat{f} = \im
    f$.  Moreover, if\/ $\im f$~is not isomorphic to~$K_{1,1}$, the
    countable universal bipartite graph~$\Gamma_{\infty}$ may be
    constructed so that there are $2^{\aleph_{0}}$ such
    extensions~$\hat{f}$.
  \item If $f$~is an automorphism of~$\Gamma$, then there is an
    automorphism~$\hat{f}$ of~$\Gamma_{\infty}$ such that the
    restriction of~$\hat{f}$ to~$\Gamma$ equals~$f$.
  \end{enumerate}
\end{lemma}

\begin{proof}
  For (i)~and~(ii), the argument is essentially the same as in
  Lemma~\ref{lem:map-extend}: having defined an extension~$f_{n}$
  of~$f$ to~$\Gamma_{n}$, we extend to $f_{n+1} \colon \Gamma_{n+1}
  \to \Gamma_{\infty}$, where $\Gamma_{n+1} = \mathcal{I}(\Gamma_{n})$
  by mapping each new vertex~$w_{n}$ to a vertex joined to every
  vertex of~$B_{n}f_{n}$ and each~$v_{n}$ to a vertex joined to every
  vertex of~$A_{n}f_{n}$.  In~$\Gamma_{\infty}$, there are always
  infinitely many choices and hence in the end we obtain
  $2^{\aleph_{0}}$~extensions~$\hat{f}$ in~(i).

  The only place where extra care is required in the argument is in
  part~(ii) since, although we can at every stage arrange that $\im
  f_{n} = \im f$, we can no longer guarantee there are infinitely many
  such extensions.  If $\im f$~is not isomorphic to~$K_{1,1}$,  there
  is some vertex~$x$ in~$\im f$ that is joined to at least two
  vertices in this image.  Suppose $x = yf$.  At each stage, when
  constructing~$\mathcal{I}(\Gamma_{n})$ we enumerate the finite
  subsets as above by always choosing $A_{1} = \{y\}$.  This ensures
  that there are at least two choices for the image of the new
  vertex~$w_{1}$, namely any of the vertices joined to~$x$.  We then
  continue as before.  We conclude that there are at least two
  extensions of the endomorphism~$f_{n}$ of~$\Gamma_{n}$ to
  endomorphisms~$f_{n+1}$ of~$\Gamma_{n+1}$ preserving the image being
  equal to~$\im f$.  Hence there are
  $2^{\aleph_{0}}$~extensions~$\hat{f}$ of~$f$ to an endomorphism
  of~$\Gamma_{\infty}$ with $\im \hat{f} = \im f$.

  Part~(iii) is proved exactly as in the case for graphs.
\end{proof}

Relying on the above lemma, but otherwise proceeding in exactly the
same way as in Section~\ref{sec:graph}, we establish parts
(i)~and~(ii) of the result that classifies images of idempotent
endomorphisms of the countable universal bipartite graph.

\begin{thm}
  \label{thm:bipartite-idemp}
  Let $\Gamma$~be a countable bipartite graph.  Then
  \begin{enumerate}
  \item there exists an idempotent endomorphism~$f$ of the countable
    universal bipartite graph~$B$ such that $\im f \cong \Gamma$ if
    and only if\/ $\Gamma$~is algebraically closed;
  \item if\/ $\Gamma$~is algebraically closed and is not isomorphic
    to the finite complete bipartite graph~$K_{1,1}$, then there are
    $2^{\aleph_{0}}$~idempotent endomorphisms~$f$ of~$B$ such that
    $\im f \cong \Gamma$;
  \item if\/ $\Gamma \cong K_{1,1}$, there are $\aleph_{0}$~idempotent
    endomorphisms~$f$ of~$B$ such that $\im f \cong \Gamma$.
  \end{enumerate}
\end{thm}

\begin{proof}
  (iii)~If $V = V_{0} \disjunion V_{1}$ is the partition of the vertex
  set associated to the partition relation on~$\Gamma$, then
  idempotent endomorphisms of~$f$ with $\im f \cong K_{1,1}$ are
  determined by choosing any edge~$(u,v)$ present in~$E$ where $u \in
  V_{0}$ and $v \in V_{1}$, and then mapping vertices in~$V_{0}$
  to~$u$ and vertices in~$V_{1}$ to~$v$.  As $B$~has $\aleph_{0}$~many
  edges, there are $\aleph_{0}$~such idempotent endomorphisms.
\end{proof}

By following the same steps as used to establish
Theorem~\ref{thm:R-reg-classes}, we can then establish parts
(i)~and~(ii) of the following.  We rely on
Theorem~\ref{thm:ac-bipartite} and note that strongly algebraically
closed bipartite graphs are certainly not isomorphic to~$K_{1,1}$.
For the final part, two idempotent endomorphisms of~$B$ with image
isomorphic to~$K_{1,1}$ are $\scD$\nbd related by
Lemma~\ref{lem:reg-classes} and each of the $\aleph_{0}$~such
idempotent endomorphisms determines a group $\scH$\nbd class.

\begin{thm}
  \label{thm:B-reg-classes}
  Let $B$~denote the countable universal bipartite graph.
  \begin{enumerate}
  \item Let\/ $\Gamma$~be a countable graph.  Then there exist
    $2^{\aleph_{0}}$~distinct regular $\scD$\nbd classes of\/~$\End B$
    whose group $\scH$\nbd classes are isomorphic to~$\Aut \Gamma$.
  \item Let $f$~be an idempotent endomorphism of~$B$ whose image is
    not isomorphic to the finite complete bipartite graph~$K_{1,1}$.
    Then the $\scD$\nbd class of~$f$ in~$\End B$ contains
    $2^{\aleph_{0}}$~distinct group $\scH$\nbd classes.
  \item There is a single $\scD$\nbd class of~$B$ containing the
    idempotent endomorphisms with image isomorphic to~$K_{1,1}$.  This
    $\scD$\nbd class contains precisely $\aleph_{0}$~group $\scH$\nbd
    classes, each of which is isomorphic to the cyclic group of
    order~$2$. \qed
  \end{enumerate}
\end{thm}

We now turn to establishing information about $\scL$- and $\scR$\nbd
classes in the endomorphism monoid of~$B$.  We deal first with the
bipartite graph version of Theorem~\ref{thm:graph-nonregLR}, since it
involves fewer changes.  We use essentially the same
construction~$\Gamma^{\sharp}$ as in Section~\ref{sec:graph}.  Let
$\Gamma = (V,E,P)$ be a bipartite graph with $P = (V_{0} \times V_{0})
\cup (V_{1} \times V_{1})$.  Assume that $V_{k} = \set{v_{i}^{(k)}}{i
  \in \N}$ for $k = 0$,~$1$.  Then define $\Gamma^{\sharp} =
(V^{\sharp}, E^{\sharp}, P^{\sharp})$, where $V^{\sharp} =
V_{0}^{\sharp} \cup V_{1}^{\sharp}$,
\begin{align*}
  V_{k}^{\sharp} &= \set{ v_{i,r}^{(k)} }{ i \in \N, \; r \in \{0,1\}
  }, \qquad \text{for $k = 0,1$}, \\
  E^{\sharp} &= \set{ (v_{i,r}^{(k)}, v_{j,s}^{(1-k)}) }{
    (v_{i}^{(k)},v_{j}^{(1-k)}) \in E, \; r,s \in \{0,1\} }, \\
  P^{\sharp} &= (V_{0}^{\sharp} \times V_{0}^{\sharp}) \cup
  (V_{1}^{\sharp} \times V_{1}^{\sharp}).
\end{align*}
The analogue of Lemma~\ref{lem:sharp} holds for this bipartite version
of~$\Gamma^{\sharp}$ and, indeed, if $\Gamma$~is strongly
algebraically closed, then so is~$\Gamma^{\sharp}$.

We now mimic the proof of Theorem~\ref{thm:graph-nonregLR}, using our
bipartite construction $\Gamma^{\sharp} = (V^{\sharp}, E^{\sharp},
P^{\sharp})$.  We need $\Gamma$~to be strongly algebraically closed in
order to be able to apply Corollary~\ref{cor:hom-from-random} to
construct the injective homomorphism $f \colon B \to \Lambda_{0}$ that
appears.  In addition, in the last half of the proof we must take $E =
(W_{0} \cup W_{1}, \emptyset, (W_{0} \times W_{0}) \cup (W_{1} \times
W_{1}) )$ as the empty bipartite graph where $W_{0}$~and~$W_{1}$ are
countably infinite disjoint sets.  This then allows us to establish:

\begin{thm}
  \label{thm:bipartite-nonregLR}
  Let $\Gamma$~be a countable strongly algebraically closed bipartite
  graph that is not isomorphic to the countable universal bipartite
  graph~$B$.  Then there exists a non-regular injective
  endomorphism~$f$ of~$R$ such that the subgraph induced by the images
  of the vertices under~$f$ is isomorphic to~$\Gamma$ and such that
  the $\scD$\nbd class of~$f$ contains $2^{\aleph_{0}}$~$\scR$- and
  $\scL$\nbd classes. \qed
\end{thm}

The bipartite analogue of Theorem~\ref{thm:graph-regLR} involves some
surprising differences and reflects the fact that there can be finite
algebraically closed bipartite graphs.

\begin{thm}
  \label{thm:bipartite-regLR}
  Let $f$~be a regular endomorphism of the countable universal
  bipartite graph~$B$.
  \begin{enumerate}
  \item If the image of~$f$ is infinite, then the $\scD$\nbd class
    of~$f$ contains $2^{\aleph_{0}}$~many $\scL$- and $\scR$\nbd
    classes.
  \item If the image of~$f$ is finite but not isomorphic to~$K_{1,1}$,
    then the $\scD$\nbd class of~$f$ contains $\aleph_{0}$~many
    $\scL$\nbd classes and $2^{\aleph_{0}}$~many $\scR$\nbd classes.
  \item If\/ $\im f \cong K_{1,1}$, then the $\scD$\nbd class of~$f$
    contains $\aleph_{0}$~many $\scL$\nbd classes and one $\scR$\nbd
    class.
  \end{enumerate}
\end{thm}

\begin{proof}
  Suppose that $B = (V,E,P)$, where $V = V_{0} \disjunion V_{1}$ is
  the partition of the vertices determined by the relation~$P$.  Let
  $\Gamma = \im f$.  For the $\scR$\nbd classes in the $\scD$\nbd
  class~$D_{f}$ of~$f$, we can, for (i)~and~(ii), argue exactly as in
  Theorem~\ref{thm:graph-regLR}: build~$B$ as~$\Gamma_{\infty}$
  around~$\Gamma$ as described above and use
  Lemma~\ref{lem:bipartite-extend}(ii) to extend the identity map
  on~$\Gamma$ to $2^{\aleph_{0}}$~idempotent endomorphisms of~$B$ with
  image~$\Gamma$.  Each such extension is $\scD$\nbd related to~$f$
  but they have distinct kernels and so are not $\scR$\nbd related to
  each other by Lemma~\ref{lem:reg-classes}.

  For the $\scR$\nbd classes in Case~(iii), note that when $\Gamma
  \cong K_{1,1}$, the endomorphisms in~$D_{f}$ map all the vertices
  in~$V_{0}$ to some fixed vertex~$v$ and all the vertices in~$V_{1}$
  to some fixed vertex~$w$ joined to~$v$ (and necessarily $v$~and~$w$
  lie in different parts of the partition).  Thus the kernel of such
  an endomorphism equals the partition relation~$P$ and we conclude
  that all endomorphisms in~$D_{f}$ are $\scR$\nbd related by
  Lemma~\ref{lem:reg-classes}(ii).

  When $\Gamma$~is infinite (i.e., Case~(i)), we build a copy of~$B$
  around the bipartite graph~$\Gamma^{\sharp}$ and use the same
  argument as in Theorem~\ref{thm:graph-regLR} to show that
  $D_{f}$~contains $2^{\aleph_{0}}$~many $\scL$\nbd classes.  When
  $\Gamma$~is finite (i.e., Cases~(ii) and~(iii)), write $n =
  \order{Vf}$.  If $g$~is $\scD$\nbd related to~$f$, then $\im g \cong
  \im f$ and so $\order{Vg} = n$.  There are $\aleph_{0}$~many subsets
  of~$V$ of cardinality~$n$ and so at most $\aleph_{0}$~many
  $\scL$\nbd classes in~$D_{f}$ by Lemma~\ref{lem:reg-classes}(i).
  However, we can construct infinitely many non-$\scL$\nbd related
  endomorphisms by adjusting the construction~$\Gamma^{\sharp}$ as
  follows.

  Assume $\Gamma = (W,F,Q)$ with associated partition $W = W_{0}
  \disjunion W_{1}$ of its vertices.  Write $W_{k} =
  \set{w_{i}^{(k)}}{i \in I_{k}}$ for finite subsets
  $I_{0}$~and~$I_{1}$ of~$\N$.  Define $\Gamma^{\natural} =
  (W^{\natural}, F^{\natural}, Q^{\natural})$, where $W^{\natural}
  = W_{0}^{\natural} \cup W_{1}^{\natural}$,
  \begin{align*}
    W_{k}^{\natural} &= \set{ w_{i,r}^{(k)} }{ i \in I_{k}, \; r \in
      \N }, \qquad \text{for $k = 0,1$,} \\
    F^{\natural} &= \set{ (w_{i,r}^{(k)},w_{j,s}^{(1-k)}) }{
      (v_{i}^{(k)},v_{j}^{(1-k)}) \in F, \; r,s \in \N }, \\
    Q^{\natural} &= (W_{0}^{\natural} \times W_{0}^{\natural}) \cup
    (W_{1}^{\natural} \times W_{1}^{\natural}).
  \end{align*}
  Thus we are now in effect replacing each edge in~$\Gamma$ by a copy
  of the infinite complete bipartite
  graph~$K_{\aleph_{0},\aleph_{0}}$.  The remainder of the argument is
  similar to Theorem~\ref{thm:graph-regLR}.  We build a copy of~$B$
  around~$\Gamma^{\natural}$ and extend the identity map
  on~$\Gamma^{\natural}$ to an idempotent endomorphism~$g$ of~$B$ with
  image~$\Gamma^{\natural}$.  For any~$\mathbf{b} = (b_{i}^{(k)})$
  with $b_{i}^{(k)} \in \N$ for each $i \in I_{k}$, we define an
  endomorphism~$\phi_{\mathbf{b}} \colon \Gamma^{\natural} \to
  \Gamma^{\natural}$ by $w^{(k)}_{i,r}\phi_{\mathbf{b}} =
  w_{i,b_{i}^{(k)}}^{(k)}$.  Then $g\phi_{\mathbf{b}}$~has image
  isomorphic to~$\Gamma$ and so is $\scD$\nbd related to~$f$, but as
  $\mathbf{b}$~varies we obtain infinitely many distinct images and so
  these endomorphisms are not $\scL$\nbd related.  This completes the
  proof.
\end{proof}

We can perform the same arguments for the \Schutz\ groups of
$\scH$\nbd classes of non-regular endomorphisms of the countable
universal bipartite graph $B = (V,E,P)$ as in
Section~\ref{sec:graph}.  If $\Gamma_{0} = (V_{0},E_{0},P_{0})$ is a
strongly algebraically closed bipartite graph, let $F_{0} \subseteq
E_{0}$ be such that $(V_{0},F_{0},P_{0}) \cong B$ (as provided by
Proposition~\ref{prop:characterise-ac}(ii)).  Assume that $B$~has been
constructed using~$\Gamma_{0}$ in the initial step of our method and
let $f \colon B \to B$ be the endomorphism that realises this
isomorphism.  Then we establish:

\begin{prop}
  \label{prop:B-Schutz}
  Let $f$~be an injective endomorphism of the countable universal
  bipartite graph~$B$ of the form specified above and $H = H_{f}$.
  Then $\mathcal{S}_{H} \cong \Aut \langle Vf \rangle \cap \Aut ( \im
  f )$. \qed
\end{prop}

To complete the work on the \Schutz\ group, we shall need a bipartite
analogue of the graphs~$M_{S}$ appearing in Section~\ref{sec:graph}.
For $S \subseteq \N \setminus \{0,1\}$, recall $L_{S}$~contains
vertices~$\ell_{n}$ (for $n \in \N$) and~$v_{n}$ (for~$n \in S$).  Let
$x_{n}$~and~$y_{n}$ (for $n \in \N$) be new vertices and set
\begin{align*}
  V_{0} &= \set{\ell_{n}}{\text{$n$~is even}} \cup
  \set{v_{n}}{\text{$n \in S$ is odd}} \cup \set{x_{n}}{n \in \N}, \\
  V_{1} &= \set{\ell_{n}}{\text{$n$~is odd}} \cup \set{v_{n}}{\text{$n
      \in S$ is even}} \cup \set{y_{n}}{n \in \N}.
\end{align*}
Let $N_{S}$~be the bipartite graph with vertex set~$V_{0} \cup V_{1}$,
partition relation $(V_{0} \times V_{0}) \cup (V_{1} \times V_{1})$
and edges consisting of all edges present in~$L_{S}$, together with an
edge between each~$x_{n}$ and every vertex in~$V_{1}$ and an edge
between each~$y_{n}$ and every vertex in~$V_{0}$.

Let $f$~be any automorphism of~$N_{S}$.  Since the $x_{n}$~and~$y_{n}$
are the only vertices adjacent to all vertices in the other part of
the partition, either $f$~fixes the parts and then must permute
the~$x_{n}$ and permute the~$y_{n}$, or $f$~interchanges the parts and
then it maps the~$x_{n}$ to the~$y_{n}$ and \emph{vice versa}.
Therefore $f$~induces an automorphism of the bipartite
graph~$L_{S}$.  Since $\Aut L_{S} = \1$, we conclude that $f$~actually
does fix the parts of the partition and simply permutes the~$x_{n}$
and permutes the~$y_{n}$.  Hence $\Aut N_{S} \cong (\Sym \N)^{2}$.
Similarly $N_{S} \cong N_{T}$ if and only if $S = T$.

Now let $\Gamma$~be an arbitrary countable (undirected and not
necessarily bipartite) graph.  Apply
Theorem~\ref{thm:bipartite-automs} to construct a countable bipartite
graph~$\Lambda$ satisfying $\Aut \Lambda \cong \Aut \Gamma$.  Let
$S_{n}$, for $n \in \N$, be a sequence of distinct subsets of~$\N
\setminus \{0,1\}$ such that the bipartite graph~$N_{S_{n}}$ is not
isomorphic to any connected component of~$\Lambda$.  (Indeed, note
that the~$\Lambda$ occurring in Theorem~\ref{thm:bipartite-automs} is
connected, so we simply require $N_{S_{n}} \not\cong \Lambda$.)
Define $\Gamma_{0}^{\ast} = \Lambda^{\ddagger}$ (the bipartite
complement of~$\Lambda$, as described above).  Then, assuming that the
bipartite graph $\Gamma_{n}^{\ast}$~has been defined with partition
relation~$(W_{0} \times W_{0}) \cup (W_{1} \times W_{1})$, enumerate
the finite subsets of~$W_{0}$ as~$(A_{i})_{i \in \N}$ and the finite
subsets of~$W_{1}$ as~$(B_{i})_{i \in \N}$.  Let the vertices
of~$\Gamma_{n+1}^{\ast}$ be the union of the vertices
of~$\Gamma_{n}^{\ast}$, the vertices of~$L_{S_{n}}$ and new
vertices~$\set{x_{i}^{(n)},y_{i}^{(n)}}{i \in \N}$.  Define the edges
of~$\Gamma_{n+1}^{\ast}$ to be the edges of~$\Gamma_{n}^{\ast}$
together with edges between $a$~and~$y_{i}^{(n)}$ for all $a \in
A_{i}$ and between $b$~and~$x_{i}^{(n)}$ for all $b \in B_{i}$.  The
partition relation on~$\Gamma_{n+1}^{\ast}$ is the one that groups
together all the vertices in~$W_{0}$ with all the~$x_{i}^{(n)}$ and
all the vertices in~$W_{1}$ with the~$y_{i}^{(n)}$.  Having
constructed the bipartite graphs~$\Gamma_{n}^{\ast}$, we let
$\Gamma^{\ast} = (V^{\ast},E^{\ast},P^{\ast})$ be the limit of this
sequence of graphs.  By construction, $\Gamma^{\ast}$~is existentially
closed and therefore isomorphic to the countable universal bipartite
graph~$B$.

Now let $\Gamma_{0} = (V^{\ast},E_{0},P^{\ast})$ be the bipartite
graph whose edges are all possible edges permitted by the bipartite
relation~$P^{\ast}$, except the following are \emph{not} included:
\begin{enumerate}
\item the edges in~$\Lambda$;
\item for each~$n \in \N$, all edges between an~$x_{i}^{(n)}$ and
  a~$y_{j}^{(n)}$;
\item for each~$n \in \N$, the edges in~$L_{S_{n}}$;
\item for each~$n \in \N$, all (permitted) edges between a vertex
  in~$L_{S_{n}}$ and a vertex $x_{i}^{(n)}$~or~$y_{i}^{(n)}$.
\end{enumerate}
As in previous sections, we have arranged that $E^{\ast} \subseteq
E_{0}$.  Therefore $\Gamma_{0}$~is algebraically closed and we use
this when applying Proposition~\ref{prop:B-Schutz}.  The endomorphism
$f \colon B \to B$ is not regular since $E_{0} \neq E^{\ast}$.

The bipartite complement~$\Gamma_{0}^{\ddagger}$ is the disjoint union
of the bipartite graphs $\Lambda$~and $N_{S_{n}}$ (for all~$n \in
\N$).  Hence
\[
\Aut \Gamma_{0} \cong \Aut \Lambda \times \prod_{n \in \N} \Aut
N_{S_{n}} \cong \Aut \Gamma \times (\Sym \N)^{\aleph_{0}}.
\]
As in the previous sections, we observe that
$\Aut(V^{\ast},E_{0},P^{\ast}) \cap \Aut(V^{\ast},E^{\ast},P^{\ast})$
is isomorphic to~$\Aut \Gamma$, and by varying~$S_{1}$ we construct
$2^{\aleph_{0}}$~many $\scD$\nbd classes of such endomorphisms~$f$.
This completes our final step in establishing the analogue of
Theorem~\ref{thm:R-uncountSchutz} for bipartite graphs.

\begin{thm}
  \label{thm:B-uncountSchutz}
  Let $\Gamma$~be any countable graph.  There are $2^{\aleph_{0}}$
  non-regular $\scD$\nbd classes of the countable universal bipartite
  graph~$B$ such that the \Schutz\ group of $\scH$\nbd classes within
  them are isomorphic to~$\Aut \Gamma$. \qed
\end{thm}

\paragraph{Acknowledgements:} Igor Dolinka acknowledges the support of
Grant No.\ 17409 of the Ministry of Education, Science, and
Technological Development of the Republic of Serbia, Grant
No.\ 1136/2014 of the Secretariat of Science and Technological
Development of the Autonomous Province of Vojvodina.  Robert Gray was
supported by an EPSRC Postdoctoral Fellowship EP/E043194/1 held at the
School of Mathematics and Statistics, University of St Andrews.
Jillian McPhee was funded by an EPSRC Doctoral Training Grant.  Martyn
Quick acknowledges support by EPSRC grant EP/H011978/1.  The authors
also thank the anonymous referee for their careful reading of the
paper and helpful comments.

\end{document}